\documentclass[11pt]{article}
\usepackage{amsmath}
\usepackage{amsfonts}
\usepackage{amssymb,hyperref}
\usepackage{xcolor}
\usepackage{graphicx}
\usepackage{enumerate}

\newtheorem{theorem}{Theorem}[section]
\newcommand{\ter}{\hspace{\stretch{3}}$\square$\\[1.8ex]}
\newtheorem{lemma}[theorem]{Lemma}

\newtheorem{proposition}[theorem]{Proposition}
\newtheorem{remark}[theorem]{Remark}

\newtheorem{definition}[theorem]{Definition}

\def\Rd{\mathbb{R}^d}

\def\R{\mathbb{R}}

\def\E{\mathbb{E}}

\def\S{\mathcal{T}}
\def\s{{\cal S}^\prime(\Rd)}
\def\ss{{\cal S}(\Rd)}
\def\J{\mathcal{J}}

\def\l{\langle}
\def\r{\rangle}
\def\Ref#1{(\ref{#1})}

\def\ph{\phantom{\sum}\!\!\!\!\!\!\!\!\!\,}

\def\LL{\left}
\def\RR{\right}

\newcommand{\proof}{{\it \bf Proof:  }}

\begin{document}
	
	\title{\bf Occupation time fluctuations of an age-dependent critical binary branching particle system}
	\author{\sc J.A. L\'{o}pez-Mimbela\thanks{Centro de Investigaci\'{o}n en Matem\'{a}ticas, Guanajuato, Mexico. {\tt jalfredo@cimat.mx}} \and {\sc A.  Murillo-Salas}\thanks{Departamento de Matem\'{a}ticas, Universidad de Guanajuato, Guanajuato, Mexico. {\tt amurillos@ugto.mx}} \and {\sc J.H. Ram\'{\i}rez-Gonz\'{a}lez}\thanks{Centro de Investigaci\'{o}n en Matem\'{a}ticas, Guanajuato, Mexico. {\tt hermenegildo.ramirez@cimat.mx}} }
	\date{ }
	\maketitle
	
	\begin{abstract}
		We study  the limit of fluctuations of the rescaled occupation time process of a branching particle system in $\Rd$, where the particles are subject to symmetric $\alpha$-stable migration ($0<\alpha\leq2$), critical binary branching, and general non-lattice lifetime distribution.  We focus on two different regimes: lifetime distributions  having  finite expectation, and Pareto-type lifetime distributions, i.e.  distributions belonging to the normal domain of attraction of a $\gamma$-stable law with $\gamma\in(0,1)$.
		In the latter case we show that, for dimensions $\alpha\gamma<d<\alpha(1+\gamma)$, the fluctuations of the rescaled occupation time converge weakly to a centered Gaussian process whose covariance function is explicitly calculated, and we call it {\em weighted sub-fractional Brownian motion.} Moreover, in the case of lifetimes with finite mean, we show that for $\alpha<d<2\alpha$ the fluctuation limit turns out to be the same as in the case of exponentially distributed lifetimes  studied by Bojdecki {\em et al.\/} \cite{BGT,BGT1,BGT2}. We also investigate the maximal parameter range allowing  existence of the weighted sub-fractional Brownian motion and provide some of its fundamental properties, such as path continuity, long-range dependence, self-similarity and the lack of Markov property.
		
		\bigskip
		
		\bigskip
		
	\end{abstract}

	\section{Introduction}
	
	Our aim in this paper is to investigate the occupation time fluctuations of a population  in $\Rd$ which evolves as follows.
	During its lifetime $S$,
	any given individual independently develops a spherically
	symmetric $\alpha$-stable process with infinitesimal generator the fractional power $\Delta_{\alpha}:=-(-\Delta)^{\alpha/2}$ of the Laplacian,
	$0<\alpha\le2$,
	and at the end of its life it
	either disappears, or  is replaced at the site where it died by two
	newborns, each event occurring with probability 1/2.  The population starts off from a Poisson random field having the Lebesgue measure $\Lambda$ as its intensity. Along with the usual independence assumptions in branching systems,  we also assume that the particle lifetimes have a general non-lattice distribution,  and that
	any individual in the initial population has age 0. We focus on two different regimes for the distribution of $S$: either $S$ has finite mean $\mu>0$ or  $S$ has a distribution function $F$ such that
	\begin{equation}\label{tail}
		F(0)=0,\quad F(x)<1 \mbox{ for all $x\geq0$},\quad \mbox{and}\quad
		1-F(t)\sim\frac{1}{t^\gamma\Gamma(1-\gamma)}\quad\mbox{as}\quad t\rightarrow\infty,
	\end{equation}
	where  $0<\gamma<1$ and $\Gamma$ denotes the usual Gamma function.
	
	Let $Z(t)$ be the counting measure in $\Rd$ whose atoms are the positions of particles alive at time $t$, and let $Z\equiv\{Z(t),\ t\geq0\}$.
	Recall that the occupation time of the
	measure-valued process $Z$ is again a measure-valued process
	$J\equiv\{J(t)$, $t\geq0\}$ which is given by
	\begin{equation*}
		\langle\varphi,J(t)\rangle:=\int_0^t\langle\varphi,Z(s)\rangle\,
		ds,\;\;t\geq0,
	\end{equation*}
	for all bounded measurable functions
	$\varphi:\mathbb{R}^d\rightarrow\mathbb{R}_+$, where the notation $\langle\varphi,\nu\rangle$ means $\int\varphi\,d\nu$.
	Following \cite{D-W} and \cite{BGT}, for each $T>0$ we introduce the rescaled occupation time process
	${J}_T(t):={J}(Tt)$ defined by
	\begin{equation*}
		\langle\varphi,{J}_T(t)\rangle=\int_0^{Tt}\langle\varphi,Z(s)\rangle
		ds=T\int_0^t\langle\varphi,Z(Ts)\rangle \,ds,\quad t\geq0,
	\end{equation*}
	and the rescaled occupation time fluctuation process
	$\{\mathcal{J}_T(t),t\geq0\}$ given by
	\begin{eqnarray*}
		\langle\varphi,\mathcal{J}_T(t)\rangle:=\frac{1}{H_T}\left(\langle\varphi,{J}_T(t)\rangle-\ph \E
		\langle\varphi,{J}_T(t)\rangle
		\right),\quad t\geq0,
	\end{eqnarray*}
	where $H_T$ is a normalization factor such that  $H_T\to\infty$ as $T\to\infty$.
	It was shown in \cite{LM-MS} that, due to criticality of the branching and  invariance of $\Lambda$ for
	the $\alpha$-stable semigroup,  $\mathbb{E}\langle\varphi,{J}_T(t)\rangle=Tt\langle\varphi,\Lambda\rangle$. Hence, the rescaled occupation time fluctuation process takes the form
	\begin{eqnarray}\label{rescaled}
		\langle\varphi,\mathcal{J}_T(t)\rangle:=\frac{1}{H_T}\left(\langle\varphi,{J}_T(t)\rangle-\ph Tt\langle\varphi,\Lambda\rangle\right),\quad t\geq0.
	\end{eqnarray}
	The Markovian case, i.e. the case of exponentially distributed particle lifetimes, has been thoroughly investigated by T. Bojdecki, L.$\:$G. Gorostiza and A. Talarczyk in a series of seminal works, see \cite{BGTWeighted,BGT,BGT1,BGT2,BGT-gamma}. Among other results, they showed that when $S$ possesses  an exponential distribution and $\alpha<d<2\alpha$, the occupation time fluctuation process, properly  rescaled, converges weakly  toward a Gaussian process in the   space $C([0,\eta],\s)$ of continuous paths $w:[0,\eta]\to\s$ for any $\eta>0$, where $\s$ denotes the space of tempered distributions, i.e. the strong dual of the space $\ss$ of rapidly decreasing smooth functions. The limit process has a simple spatial structure whereas the temporal structure is characterized by that of {\it sub-fractional Brownian motion} (sub-fBm), i.e. a continuous centered Gaussian process $\{\zeta_t,\ t\geq0\}$ with covariance function
	\begin{equation}\label{Cv_sfBm}
		\mathcal{C}(s,t):=s^{h}+t^{h}-\frac{1}{2}\left[(s+t)^{h}+|s-t|^{h}\right],\quad s,t\geq0,
	\end{equation}
	with $h=3-d/\alpha$ ($h\in(1,2)$); see \cite{BGT1}.  According to \cite{BGT}, sub-fBm exists for all $h\in(0,2)$. For $h\neq1$ this process does not have stationary independent increments, but possesses the so-called long-range dependence property, and for $h=1$ it reduces to  Brownian motion.
	
	It is known \cite{Athreya-Ney} that the process $Z$ fails to be Markovian if $S$ does not have an exponential distribution. There are relatively few publications on models related to non-Markovian spatial branching systems. Laws of large numbers for the occupation times of $Z$ have been investigated in \cite{M} and \cite{LM-MS}. Diffusion limit-type approximations for branching systems with non-exponential particle lifetimes were developed in \cite{FVW} and \cite{KS}. Existence of a non-trivial equilibrium distribution for such kind of models was studied in \cite{VW}.
	
	Assume that $F$ is a general absolutely continuous function obeying (\ref{tail}). In this paper we prove that for dimensions  satisfying $\alpha\gamma<d<\alpha(1+\gamma)$, the occupation time fluctuation limit exists and is a centered Gaussian process  whose covariance function has a simple spatial structure, but its temporal structure is dictated, for the case $d\neq\alpha$, by a fractional noise with covariance function
	\begin{equation}\label{Wsubfbm0}
		{Q}(s,t):=\LL(\frac{d}{\alpha}-1\RR)^{-1}\int_0^{s\wedge t}r^{\gamma-1}\left[(s-r)^{2-d/\alpha}+( t-r)^{2-d/\alpha}-(t+s-2r)^{2-d/\alpha}\right]dr,\quad s,t\ge0,
	\end{equation}
	whereas for the case $d=\alpha$,  the limit  is a  centered Gaussian process whose  covariance function has a temporal structure determined by
	\begin{equation*}
		K(s,t):=\int_{0}^{s\wedge t}r^{\gamma-1}\left[(s+t-2r)\ln(s+t-2r)-(s-r)\ln(s-r)-
		\ph(t-r)\ln(t-r)\right]\,dr;
	\end{equation*}
	see Theorem \ref{main1} below.
	The special but important case of particle lifetimes with finite mean is dealt with in Theorem \ref{main}, where we show that for dimensions satisfying  $\alpha< d < 2\alpha$  the limit process is  centered Gaussian, with covariance function of the form (\ref{Cv_sfBm}). Hence, Theorem \ref{main} extends  Theorem 2.2 in \cite{BGT1}  to the case of non-exponential particle lifetimes with finite mean. Moreover,
	in this case the effect of the lifetime distribution becomes apparent only through its mean.
	
	To obtain these results we follow the method of proof used in \cite{BGT1}, i.e. the space-time random field weak convergence approach developed in \cite{BGR}, combined with the Feynman-Kac formula. However the adaptation to our case of such method is far from being straightforward.   Besides the lack of Markov property of $Z$, in our more general scenario the use of a Feynman-Kac formula is much more involved than in \cite{BGT1} due to the fact that the renewal function  associated to $F$
	is in general nonlinear, in contrast to the linear renewal function of exponential lifetimes.
	
	Notice that the  function (\ref{Wsubfbm0}) is a special case of the function ${Q}_{a,b}$ given by \begin{equation}\label{Wsubfbm1}
		{Q}_{a,b}(s,t):=\frac{1}{1-b}\int_0^{s\wedge t}r^{a}\left[(s-r)^{b}+( t-r)^{b}-(t+s-2r)^{b}\right]dr, \quad s,t\geq0,\quad a,b\in\R,
	\end{equation}
	and that, for $a=0$, (\ref{Wsubfbm1}) is the covariance function of the sub-fractional Brownian motion for $|b|<1$.
	Several other interesting cases arise as special instances of (\ref{Wsubfbm1}); see Remark \ref{variaciones} bellow. This motivated our second goal in this paper, which  is to determine suitable values  of the parameters $a,b\in\R$ for which ${Q}_{a,b}$ is a covariance function.
	It turns out that,  if the parameters $a,b$ are restricted to the domains $a>-1$ and $b\in[0,2]$ with $b\neq1$, or $a>-1$ and $-1<b<0$ with $a+b+1\geq  0$,   the function  ${Q}_{a,b}$ is positive definite; see  Theorem \ref{QgeneralTh} below.  A centered real-valued Gaussian process with covariance function (\ref{Wsubfbm1}) will be called {\it weighted sub-fractional Brownian motion}, in analogy to  the weighted fractional Brownian motion introduced in \cite{BGTWeighted}.  We recall that a weighted fractional Brownian motion is a centered Gaussian  process $\eta:=\{\eta(t),\ t\geq0\}$  with covariance function of the form
	\begin{equation}\label{W-fractinal-cov}
		H_{a,b}(s,t):= \int_{0}^{s\wedge t}r^{a}\left[(s-r)^{b}+(t-r)^{b}\right]dr,\quad s,t\geq0,
	\end{equation}
	for $a>-1$,  $-1<b\leq 1$ and $|b|\leq 1+a$; see  \cite[Thm. 2.1]{BGTWeighted}.
	In Theorem \ref{limit-properties} we show  that any weighted sub-fractional Brownian motion $\{\varsigma(t)$, $t\ge0\}$ possesses long memory (also called long-range dependence), in the sense that \[\E\left[(\varsigma(t+T)-\ph\varsigma(s+T))(\varsigma(v)-\varsigma(r))\right]\sim T^{b-2}\frac{b}{(a+1)(a+2)}(t-s)(v^{a+2}-r^{a+2})\quad\mbox{as}\quad T\rightarrow\infty.\]
	It is worth to mention that the weighted fractional Brownian motion $\eta$ also exhibits the long-range dependence property. In this case, \[\E\left[(\eta(t+T)-\ph\eta(s+T))(\eta(v)-\eta(r))\right]\sim T^{b-1}\frac{b}{a+1}(t-s)(v^{a+1}-r^{a+1})\quad\mbox{as}\quad T\rightarrow\infty;\]
	see  \cite{BGTWeighted}.

	The rest of the paper is organized as follows. In Section \ref{mainproof} we state the main results in this paper. In Section \ref{LP} we prove a recursive relation for the Laplace functional of the branching particle system which we will need in the sequel, and that  it might be of interest on its own right. Section \ref{PROOFS} is devoted to the proof of our main results.
	%

	\section{Main results } \label{mainproof} Recall that we restrict ourselves to a particle system $Z$ whose branching mechanism is critical binary,  as described in the previous section.
	In
	what follows, the symbol $\Rightarrow$ denotes weak convergence.
	
	\subsection{Fluctuation limit theorems}
	
	\begin{theorem} \label{main1} Let $F$ be an absolutely continuous lifetime distribution function satisfying (\ref{tail}). 
		Let  $\alpha\gamma<d<\alpha(1+\gamma)$ and $H_T=T^{(2+\gamma-d/\alpha)/2}$. Then $\J_T\Rightarrow \J$ in $C([0, \Upsilon],\s)$ as $T\rightarrow\infty$ for any $\Upsilon>0$, where $\{\J(t),\ t\geq0\}$ is a centered Gaussian process whose covariance function is given in the following way:
		\begin{enumerate}[(i)]
			\item For $d\neq\alpha$,
			\begin{eqnarray*}
				\mbox{\rm Cov}(\l\varphi,\J(s)\r,\l\psi,\J(t)\r)=\LL[\frac{\gamma\langle\varphi,\lambda\rangle \langle\psi,\lambda\rangle}{\Gamma(\gamma+1)(2\pi)^d(2-\frac{d}{\alpha})}\int_{\Rd}e^{-|y|^\alpha}dy\RR]{Q}(s,t),\quad s,t\geq0,
			\end{eqnarray*}
			where $\varphi,\psi\in \ss$ and
			\begin{equation}\label{Wsubfbm}
				Q(s,t)=\left(\frac{d}{\alpha}-1\right)^{-1}\int_0^{s\wedge t}r^{\gamma-1}\left[(s-r)^{2-d/\alpha}+( t-r)^{2-d/\alpha}\ph-(t+s-2r)^{2-d/\alpha}\right]dr.
			\end{equation}
			\item For $d=\alpha$,
			\begin{eqnarray*}
				\mbox{\rm Cov}(\l\varphi,\J(s)\r,\l\psi,\J(t)\r)=\LL[\frac{\gamma\langle\varphi,\lambda\rangle \langle\psi,\lambda\rangle}{\Gamma(\gamma+1)(2\pi)^d}\int_{\Rd}e^{-|y|^\alpha}dy\RR]K(s,t),\quad s,t\geq0,
			\end{eqnarray*}
			where $\varphi,\psi\in \ss$ and
			\begin{equation*}
				K(s,t):=\int_{0}^{s\wedge t}r^{\gamma-1}\left[(s+t-2r)\ln(s+t-2r)\ph-(s-r)\ln(s-r)-(t-r)\ln(t-r)\right]dr.
			\end{equation*}
		\end{enumerate}
	\end{theorem}

	\begin{theorem} \label{main} Let $F$ be an absolutely continuous lifetime distribution function with finite mean $\mu>0$. Let  $\alpha<d<2\alpha$ and $H_T=T^{(3-d/\alpha)/2}$. Then $\J_T\Rightarrow \J$ in $C([0,\Upsilon],\s)$ as $T\rightarrow\infty$ for any $ \Upsilon>0$, where $\{\J(t),\ t\geq0\}$ i a centered Gaussian process with covariance function
		\begin{eqnarray*}
			\mbox{\rm Cov}(\l\varphi,\J(s)\r,\l\psi,\J(t)\r)=\frac{\langle\varphi,\Lambda\rangle\langle\psi,\Lambda\rangle \Gamma(2-h)}{2^{d-1}\pi^{d/2}\mu\alpha \Gamma(d/2)h(h-1)}\mathcal{C}(s,t),\quad s,t\geq0,
		\end{eqnarray*}
		where $h=3-d/\alpha$, $\varphi,\psi\in \ss$ and ${\cal C}(s,t)$ is given by (\ref{Cv_sfBm}).
	\end{theorem}
	
	\begin{remark}
		The case of ``large" dimensions, i.e.  $d\geq 2\alpha$ for lifetimes with finite mean, and $d\geq \gamma(1+\alpha)$ for heavy-tailed lifetimes, are part of an ongoing research project.  Presently we  can mention that, in the case of finite mean, we get the same  Theorem 2.2 of \cite{BGT2}. On the other hand, for the case of $d\geq \alpha(1+\gamma)$  we get a very different behavior, compared to the finite mean case. In particular, for $d=\alpha(1+\gamma)$  we have that the covariance function of the fluctuations limit  has a temporal structure  of the form 
		\[ C_1 Q_\gamma (s,t) +C_1 (s\wedge t),
		\]
		where $C_1$ and $C_2$ are positive constants and $Q_\gamma(s,t)=(s\wedge t)^\gamma$, for $\gamma\in (0,1)$. Whereas, for $d>\alpha(1+\gamma)$ we only have the Brownian part, as in the case of finite mean. 
		
	\end{remark} 
	
	\begin{remark}

		It is not too difficult to see that our arguments to prove Theorem \ref{main1} can be adapted to the setting of the high density limit in \cite{BGTH};  see Section \ref{PROOFS} below. In particular, with $H_T^{2}=F_TT^{2+\frac{\gamma}{2}-\frac{d}{\alpha}}$ where $F_T \overset{T\to \infty}{\to} \infty$ and $\lim_{T \to \infty}F_T^{-1}T^{\gamma-\frac{d}{\alpha}}=0$, we can prove a result parallel to Theorem 2.2 in \cite{BGTH}, for $\beta=1$. Thus, under the assumption $d\leq \alpha\gamma $ we will have the same limit as in Theorem \ref{main1} (i). That is, the temporary structure of the occupation time fluctuations has as its limit a weighted fractional Brownian motion with parameters $a=\gamma-1$ and $b=2-
		\frac{d}{\alpha} \in (1,2)$.
		
	\end{remark}

	\subsection{Weighted sub-fractional Brownian motion}\label{W_section}
	
	In this section we give conditions on the parameters   $a$ and $b$, under which the function $Q_{a,b}$ given in (\ref{Wsubfbm1}) is a covariance function. Moreover, when
	$Q_{a,b}$ is a covariance
	we provide several properties of the associated  centered Gaussian process.  In addition, we introduce the notion of {\em weighted sub-fractional Brownian motion}.

	\begin{theorem}\label{QgeneralTh} For  $a,b>-1$ with $b\neq 1$, the function
		\begin{equation}\label{Qgeneral}
			Q_{a,b}(w,z):=\frac{1}{1-b}\int_{0}^{z\wedge w}s^{a}\left[(z-s)^{b}+(w-s)^{b}-(w+z-2s)^{b}\right]ds,\quad w,z\geq0,
		\end{equation}
		is positive definite in the following cases:
		\begin{itemize}
			\item[(i)] $a>-1$ and $0\leq b\leq 2$.
			\item[(ii)]  $a>-1$ and $-1<b<0$ with $a+b+1\geq  0$.
		\end{itemize}
	\end{theorem}

	\begin{remark}\label{variaciones}Let us mention several known instances of $Q_{a,b}$ given in (\ref{Qgeneral}):
		\begin{enumerate}
			\item[(a)] Theorem \ref{main1} (i) yields that the function \begin{equation}\label{W-Q} (w,z)\mapsto\int_{0}^{z\wedge w}s^{a}\left[(z-s)^{b}+(w-s)^{b}-(w+z-2s)^{b}\right]ds,\quad w,z\geq0,\end{equation} with $a=\gamma-1$ and $b=2-d/\alpha$, appears as the temporal structure of the covariance function of
			the rescaled occupation time fluctuation limit for a branching particle system in $\Rd$ with $\alpha$-stable motions  and lifetimes having a Pareto tail distribution (\ref{tail}).
			
			\item[(b)] In \cite{BGT-gamma} Bojdecki {\em et al.} investigated the limit fluctuations of a rescaled occupation time process of a branching particle system with particles moving according to $d$-dimensional  $\alpha$-stable motion,  starting with an inhomogeneous Poisson population with intensity measure $dx/(1+|x|^\gamma)$, where $\gamma>0$. In this case, for  $\gamma<d<\alpha$ (hence $d=1$) and normalization $T^{1-(d+\gamma)/2\alpha}$, the limit of the occupation time fluctuations  is  a  Gaussian process whose temporal structure is determined by
			the covariance function
			\begin{eqnarray}\label{W_fractional}
				C_{a,b}(w,z):=\int_{0}^{z\wedge w}s^{a}\left[(z-s)^{b}+(w-s)^{b}\right]ds,\quad w,z\geq0,
			\end{eqnarray}
			for $a=-\gamma/\alpha$ and $b=1-d/\alpha$; see \cite[Thm. 2.2]{BGT-gamma}. Latter on, the same authors proved that the maximal range of values of parameters $a,b$ that makes  (\ref{W_fractional}) a covariance function is $a>-1$, $-1<b\leq 1$ and $|b|\leq 1+a$.  The authors named the  centered Gaussian process with covariance function (\ref{W_fractional}), {\em weighted fractional Brownian motion with parameters $a$ and $b$}, see  \cite{BGTWeighted}. Notice that both, (\ref{W-Q}) and (\ref{W_fractional}), are  weighted covariance kernels, corresponding respectively to weighted sub-fractional Brownian motion, and weighted fractional Brownian motion.

			\item[(c)]  From (\ref{Qgeneral}) it follows that
			\begin{equation}\label{Q_a.cero}
				Q_{0,b}(w,z)=\frac{1}{(b+1)(1-b)}\left(w^{b+1}+z^{b+1}-\frac{1}{2}\left((w+z)^{b+1}+|w-z|^{b+1}\right)\right)
				,\quad w,z\geq0.
			\end{equation}
			Thus, modulus a constant  factor, (\ref{Q_a.cero}) coincides with the covariance function (1.4) in \cite{BGTWeighted}, therefore (\ref{Q_a.cero}) is a covariance function for $-1<b\leq 3$. In particular, for $|b|<1$ it is the covariance function of the sub-fractional Brownian motion.
			
		\end{enumerate}
	\end{remark}
	The next result exhibits a range of  parameters $a$ and $b$ for which the function $Q_{a,b}(\cdot,\cdot)$ is not a covariance function.
	\begin{lemma}\label{QnotConvariance} The function $Q_{a,b}(\cdot,\cdot)$ is not a covariance function in the following cases:
		\begin{itemize}
			\item[(i)]  $a>-1$ and $-1<b<0$, with $a+b+1< 0$;
			\item[(ii)]  $a>-1$ and $b>a+3$.
		\end{itemize}
		
	\end{lemma}
	\begin{remark}We were unable to determine  whether (\ref{Qgeneral}) is positive definite for $a>-1$ and $2<b\leq a+3$. This case remains as a challenge for future work.
	\end{remark}
	\begin{definition}\label{W_subfractional} A centered, real-valued Gaussian process $\zeta=\{\zeta_t$, $t\geq0\}$ with covariance function (\ref{Qgeneral}), whose parameters $a$ and $b$ satisfy the conditions given in Theorem \ref{QgeneralTh}, will be called {\em weighted sub-fractional Brownian motion with  parameters $a$ and $b$}.
	\end{definition}
	
	\begin{theorem}\label{limit-properties}
		Let $\zeta$ be the weighted sub-fractional Brownian motion with parameters $a$ and $b$.
		\begin{enumerate}
			\item[(i)]  $\zeta$ is a self-similar process of index $(a+b+1)/2$, i.e. for any  $c>0$,  $$\left(\zeta(ct)\right)_{t\geq 0}\overset{d}{=}\left(c^{{(1+b+a)}/{2}}\zeta(t)\right)_{t\geq 0}.$$
			
			\item[(ii)] (a) Assume that $-1<a\leq 0$, $b\in (0,1)\cup (1,2]$ and $0<a+b+1\leq 2$. {\color{black}For any $M>0$}, there exists a constant $\kappa >0$ such that
			\begin{equation}
				\E\left[(\zeta(t)-\zeta(s))^{2}\right]\leq \kappa|t-s|^b,\quad 0\le s,t<{\color{black}M}, \text{ with } 0\leq |t-s|\leq 1.
			\end{equation}
			In particular,  due to Kolmogorov's continuity theorem,  $\zeta$ possesses a  continuous version  whose paths are a.s. locally-Hölder continuous with index  $\delta$, for any $0<\delta<b/2$.
			
			(b) Assume that $-1<b\leq 0$  and $b+a>0$. There exists a constant $\kappa>0$ such that
			\begin{equation}
				\E\left[(\zeta(t)-\zeta(s))^{2}\right]\leq \kappa |t-s|^{b+1},\quad 0\le s,t<\infty.
			\end{equation}
			In particular,    $\zeta$ possesses a  continuous version  whose paths  are a.s. locally-Hölder continuous with index  $\delta$, for any $0<\delta<(b+1)/2$.

			\item[(iii)] For $0\leq r<v\leq s<t$ there holds
			\begin{eqnarray*}\lefteqn{\mathcal{Q}(r,v,s,t):=
					\E\left[(\zeta(t)-\zeta(s))(\zeta(v)-\zeta(r))\right]}\\ &=&Q_{a,b}(t,v)-Q_{a,b}(t,r)-Q_{a,b}(s,v)+Q_{a,b}(s,r)
				\\ &=&\frac{1}{1-b}\left[\int_{r}^{v}u^{a}\left((t-u)^{b}-(s-u)^{b}\right)du
				+\int_{0}^{r}u^{a}\left((t+r-2u)^{b}-(s+r-2u)^{b}\right)du\right.\\
				&& \, \phantom{MMMMMMMMMMMMMMM}
				-\left.\int_{0}^{v}u^{a}\left((t+v-2u)^{b}-(s+v-2u)^{b}\right)du\right].
			\end{eqnarray*}
			\item[(iv)] (Long-range dependence) For $b\in(0,1)\cup(1,2)$ or for $-1<b\leq 0$ with $a+b+1\geq 0$ and   $0\leq r<v\leq s<t$,
			\begin{equation}\label{LRD-identity}
				\lim_{T\rightarrow \infty}T^{2-b}\mathcal{Q}(r,v,s+T,t+T)=\frac{b}{(a+1)(a+2)}(t-s)(v^{a+2}-r^{a+2}).
			\end{equation}
			\item[(v)] $\zeta$ is not a Markov process.
		\end{enumerate}
	\end{theorem}
	%
	%
	
	\begin{theorem}\label{limit-Bm} Let $\{\zeta(t)$, $t\ge0\}$ be  the weighted subfractional Brownian motion with parameters $a$ and $b$.
		\begin{enumerate}
			\item[(i)] Let $b\in (1,2].$  {\color{black} The finite dimensional distributions of the processes} $\LL\{T^{-\frac{a+b-1}{2}}(\zeta(t+T)-\zeta(T)),\ t\geq 0\RR\}$ converge, as $T\to\infty$, to those of the process
			$\LL\{2^{b-2}b\mathcal{B}(a+1,b-1) \xi(t),\ t\geq0\RR\}$, where $\LL\{\xi(t),\ t\geq 0\RR\}$ is a weighted fractional  Brownian motion with covariance function $H_{0,1}(s,t)$ given in (\ref{W-fractinal-cov}), and ${\cal B}(x,y)$ is the Beta function.

			\item[(ii)] 
			
			Let $b\in (-1,1)$  with $a+b+1> 0$. {\color{black} The finite dimensional distributions of the processes}  $\LL\{T^{-\frac{a}{2}}(\zeta(t+T)\right.$ $\left.-\!\!\!\!\!\phantom{\int_1^2}\zeta(T)),\ t\geq 0\RR\}$ converge, as $T\to\infty$, to those of 
			$\LL\{\frac{1}{(1-b)(b+1)} X(t),\ t\geq0\RR\}$, where $\LL\{X(t),\ t\geq 0\RR\}$ is a fractional Brownian motion with Hurst parameter $(b+1)/2$.

		\end{enumerate}
	\end{theorem}

	\section{Laplace functional}\label{LP}
	In this section we will compute the Laplace functional of the occupation time process of $Z$ in a general setting, i.e., we only assume that the branching law is characterized by its probability generating function $h(s)=\sum_{k=0}^\infty p_ks^k$, $|s|\leq1$, and the particle lifetimes by a general distribution function $F$ with support in $[0,\infty).$ The symmetric $\alpha$-stable motion in $\Rd$ will be denoted by $\xi=\{\xi_t,\ t\geq0\}$ and by $\S=\{\S_t,\ t\geq 0\}$ its semigroup.
	
	By definition  $Z_t(A)$ is the number of individuals living in $A\in\mathcal{B}(\Rd)$ at time $t\ge0$,
	where $\mathcal{B}(\Rd)$ denotes the system of Borel set in $\Rd$.
	Let
	$\{S_k,\ k\geq1\}$ be a sequence of i.i.d. random variables
	with common distribution function $F$, and let
	\[
	N_t=\sum_{k=1}^\infty1_{\{W_k\leq t\}}\quad\mbox{and}\quad U(t)=\sum_{n=1}^\infty F^{*n}(t),\quad t\geq0,
	\]
	be the respective renewal process and renewal function,  where the random sequence $\{W_k,\ k\geq0\}$ is defined recursively by
	$$
	W_0=0,\quad W_{k+1}=W_k+S_{k+1},\quad k\geq0.$$
	Define $g(s):=h(1-s)-(1-s)$, $|s|\leq1$. Notice that in the case of critical binary branching $h(s)=s+\frac{1}{2}(1-s)^2$ and  $g(s)=\frac{1}{2}s^2$.
	
	Now, for any nonnegative   $\Psi\in {\cal S}(\R^{d+1})$, we define the function
	\begin{equation}\label{LaplaceV}
		v_\Psi(x,r,t):=\E_x\left[1-e^{-\int_0^t\langle\Psi(\cdot,s+r),Z_s\rangle \,ds}\right],\quad x\in\Rd,\quad r,t\geq0,
	\end{equation}
	where $\E_x$ denotes the expectation operator in a population starting with one particle of age $0$, located at the position $x\in\Rd$.
	\begin{proposition} The function $v_\Psi(x,r,t)$ satisfies the  integral equation
		\begin{equation}\label{LaplaceV1}
			v_\Psi(x,r,t)=\E_x\left[1-e^{-\int_0^t\Psi(\xi_s,r+s)\,ds}\right]
			-\int_0^t\E_x\left[e^{-\int_0^{u}\Psi(\xi_s,r+s)\,ds}g\LL(v_\Psi(\xi_{u},\ph r+u,t-u)\RR) \right]d U(u).
		\end{equation}
	\end{proposition}
	\proof  Formula (\ref{LaplaceV1}) obviously holds for $t=0$. Let $t>0$. By conditioning on the first branching time we get,
	$$
	1-v_\Psi(x,r,t)=\E_x\left[e^{-\int_0^t\Psi(\xi_s,r+s)ds}\mathbf{1}_{\{S_1>t\}}\right]
	+\E_x\left[e^{-\int_0^{S_1}\Psi(\xi_s,r+s)ds}h\left(1-v_\Psi\LL(\xi_{S_1},\ph r+S_1,t-S_1\RR) \right)\mathbf{1}_{\{S_1\leq t\}}\right],
	$$
	or equivalently,
	\begin{equation}\label{9f}
		\begin{split}
			v_\Psi(x,r,t)=&\E_x\Bigg[\Bigg(1-e^{-\int_0^t\Psi(\xi_s,r+s)ds}\Bigg)\mathbf{1}_{\{S_1>t\}}+\Bigg(1-e^{-\int_0^{S_1}\Psi(\xi_s,r+s)ds}\Bigg)\mathbf{1}_{\{S_1\leq t\}}\\
			&-e^{-\int_0^{S_1} \Psi(\xi_s,r+s)ds}g(v_\Psi(\xi_{S_1},r+S_1,t-S_1))\mathbf{1}_{\{S_1\leq t\}}\Bigg]\\
			&+\E_x\Bigg[e^{-\int_0^{S_1} \Psi(\xi_s,r+s)ds}v_\Psi(\xi_ {S_1},r+S_1,t-S_1)\mathbf{1}_{\{S_1\leq t\}}\Bigg].
		\end{split}
	\end{equation}
	Next, we consider the event $[S_1\leq t]$ and write $\xi^x=\{\xi^x_s,\ s\ge0\}$ for a symmetric $\alpha$-stable motion starting in $x\in\Rd$. Proceeding as  above with $r$, $t$ and $x$ replaced respectively by  $r+S_1$,  $ t-S_1$ and $\xi_{S_1}$,   and designating $\E_{\xi_{S_1}}(\cdot)$ the expected value starting with a particle at position $\xi_{S_1}$, given the $\sigma-$algebra $\sigma((\xi_{s})_{0\leq s\leq S_1}\cup S_1)$, we obtain
	\begin{eqnarray*}
		\lefteqn{
			v_\Psi(\xi_{S_1}^{x},r+S_1,t-S_1)1_{\{S_1\leq t\}}}
		\\
		&= &
		\E_{\xi_{S_1}}\LL[\LL(1-e^{-\int_{0}^{t-S_1}\Psi\LL(\xi_{u}^{\xi_{S_1}^{x}},r+S_1+u\RR)\,du}\RR)1_{\{W_1\leq t<W_2\}}\RR] \\ &&
		+\E_{\xi_{S_1}}\LL[\LL(1-e^{-\int_{0}^{S_2}\Psi\LL(\xi_{u}^{\xi_{S_1}^{x}},r+S_1+u\RR)\,du}\RR)1_{\{W_2\leq t\}}\RR]\\ &&+
		\E_{\xi_{S_1}} \Bigg[-e^{-\int_0^{S_2}\Psi\LL(\xi_{u}^{\xi_{S_1}^{x}},r+S_1+u\RR)\,du}g\LL(v_\Psi\LL(\xi_{S_2}^{\xi_{S_1}^{x}},\ph r+S_1+S_2,t-S_1-S_2\RR)\RR)1_{\{W_2\leq t\}}\Bigg]\\
		&&+\E_{\xi_{S_1}}\Bigg[e^{-\int_{0}^{S_2}\Psi\LL(\xi_{u}^{\xi_{S_1}^{x}},r+S_1+u\RR)\,du}v_\Psi\LL(\xi_{S_2}^{\xi_{S_1}^{x}},r+S_1+S_2,t-S_1-S_2\RR)1_{\{W_2\leq t\}}\Bigg].
	\end{eqnarray*}
	Hence, by the strong Markov property of $\{\xi_{s},\ s\geq 0\}$,
	\begin{eqnarray} \nonumber
		\lefteqn{\E_x\LL[e^{-\int_0^{S_1}\Psi(\xi_u^{x},r+u)du}v_{\Psi}(\xi_{S_1}^{x},r+S_1,t-S_1)1_{\{S_1\leq t\}}\RR]}\\ \nonumber
		&& \\ \nonumber
		&=&\E_x\LL[-e^{-\int_{0}^{S_1+S_2}\Psi\LL(\xi_{u}^{x},r+u\RR)\,du}g(v_\Psi(\xi_{S_1+S_2}^{x},r+S_1+S_2,t-S_1-S_2))1_{\{W_2\leq t\}}\RR]\\ \nonumber  && \\ \nonumber
		&&+\E_x\LL[\LL(e^{-\int_{0}^{S_1}\Psi(\xi_{u}^{x},r+u)du}-e^{-\int_{0}^{S_1}\Psi(\xi_{u}^{x},r+u)du-\int_{0}^{t-S_1}\Psi(\xi_{S_1+u}^{x},r+S_1+u)du}\RR)1_{\{W_1\leq t<W_2\}}\RR]\\ \nonumber  && \\ \nonumber
		&&+\E_x\LL[\LL(e^{-\int_{0}^{S_1}\Psi(\xi_{u}^{x},r+u)du}-e^{-\int_{0}^{S_1}\Psi(\xi_{u}^{x},r+u)du-\int_{0}^{S_2}\Psi(\xi_{S_1+u}^{x},r+S_1+u)du}\RR)1_{\{W_2\leq t\}}\RR]\\ \nonumber  && \\ \nonumber
		&&+\E_x\LL[e^{-\int_{0}^{S_1}\Psi(\xi_{u}^{x},r+u)du-\int_{0}^{S_2}\Psi(\xi_{S_1+u}^{x},r+S_1+u)du}v_\Psi(\xi_{S_1+S_2}^{x},r+S_1+S_2,t-S_1-S_2)1_{\{W_2\leq t\}}\RR]\\ \nonumber  && \\ \nonumber
		&=&\E_x\LL[-e^{-\int_{0}^{S_1+S_2}\Psi(\xi_{u}^{x},r+u)du}g(v_{\Psi}(\xi_{S_1+S_2}^{x},r+S_1+S_2,t-S_1-S_2))1_{\{W_2\leq t\}}\RR]\\ \nonumber  && \\ \nonumber
		&&+\E_x\LL[\LL(e^{-\int_{0}^{S_1}\Psi(\xi_{u}^{x},r+u)du}-e^{-\int_{0}^{t}\Psi(\xi_{u}^{x},r+u)du}\RR)1_{\{W_1\leq t<W_2\}}\RR]\\ \nonumber  && \\ \nonumber
		&&+\E_x\LL[\LL(e^{-\int_{0}^{S_1}\Psi(\xi_{u}^{x},r+u)du}-e^{-\int_{0}^{S_1+S_2}\Psi(\xi_{u}^{x},r+u)du}\RR)1_{\{W_2\leq t\}}\RR]\\ \nonumber  && \\ \nonumber
		&&+\E_x\LL[e^{-\int_{0}^{S_1+S_2}\Psi(\xi_{u}^{x},r+u)du}v_{\Psi}(\xi_{S_1+S_2}^{x},r+S_1+S_2,t-S_1-S_2)1_{\{W_2\leq t\}}\RR]\\ \nonumber  && \\ \nonumber
		&=&\E_x\LL[-e^{-\int_{0}^{S_1+S_2}\Psi(\xi_{u}^{x},r+u)du}g(v_{\Psi}(\xi_{S_1+S_2}^{x},r+S_1+S_2,t-S_1-S_2))1_{\{W_2\leq t\}}\RR]\\ \nonumber  && \\ \nonumber
		&&+\E_x\LL[\LL(1-e^{-\int_{0}^{t}\Psi(\xi_{u}^{x},r+u)du}\RR)1_{\{W_1\leq t<W_2\}}\RR]\\ \nonumber  && \\ \nonumber
		&&+\E_x\LL[\LL(1-e^{-\int_{0}^{S_1+S_2}\Psi(\xi_{u}^{x},r+u)du}\RR)1_{\{W_2\leq t\}}\RR]\\ \nonumber   && \\ \nonumber
		&&+\E_x\LL[\LL(e^{-\int_{0}^{S_1}\Psi(\xi_{u}^{x},r+u)du}-1\RR)1_{\{W_1\leq t\}}\RR]  \\  \nonumber && \\  \label{p2}
		&&+\E_x\LL[e^{-\int_{0}^{S_1+S_2}\Psi(\xi_{u}^{x},r+u)du}v_{\Psi}(\xi_{S_1+S_2}^{x},r+S_1+S_2,t-S_1-S_2)1_{\{W_2\leq t\}}\RR].
	\end{eqnarray}
	Plugging (\ref{p2}) into (\ref{9f}) we get
	\begin{equation}\label{p1}
		\begin{split}
			v_{\Psi}(x,r,t)&=\E_x\Bigg[-e^{\int_{0}^{S_1}\Psi(\xi_u^{x},r+u)du})g(v_{\Psi}(\xi_{S_1}^{x},r+S_1,t-S_1))1_{\{W_1\leq t\}}\Bigg]\\ &+\E_{x}\Bigg[\Bigg(1-e^{-\int_0^{S_1}\Psi(\xi_u^{x},r+u)du}\Bigg)1_{\{W_1\leq t\}}\Bigg]\\
			&+\E_x\Bigg[\Bigg(1-e^{-\int_0^{t}\Psi(\xi_u^{x},r+u)du}\Bigg)1_{\{W_1> t\}}\Bigg]\\
			&+\E_x\Bigg[-e^{-\int_{0}^{S_1+S_2}\Psi(\xi_{u}^{x},r+u)du}g(v_{\Psi}(\xi_{S_1+S_2}^{x},r+S_1+S_2,t-S_1-S_2))1_{\{W_2\leq t\}}\Bigg]\\
			&+\E_x\Bigg[\Bigg(1-e^{-\int_{0}^{t}\Psi(\xi_{u}^{x},r+u)du}\Bigg)1_{\{W_1\leq t<W_2\}}\Bigg]\\
			&+\E_x\Bigg[\Bigg(1-e^{-\int_{0}^{S_1+S_2}\Psi(\xi_{u}^{x},r+u)du}\Bigg)1_{\{W_2\leq t\}}\Bigg]\\
			&+\E_x\Bigg[\Bigg(e^{-\int_{0}^{S_1}\Psi(\xi_{u}^{x},r+u)du}-1\Bigg)1_{\{W_1\leq t\}}\Bigg]\\
			&+\E_x\Bigg[e^{-\int_{0}^{S_1+S_2}\Psi(\xi_{u}^{x},r+u)du}v_\Psi(\xi_{S_1+S_2}^{x},r+S_1+S_2,t-S_1-S_2)1_{\{W_2\leq t\}}\Bigg]\\
			&=\E_{x}\Bigg[\Bigg(1-e^{-\int_0^{t}\Psi(\xi_u^{x},r+u)du}\Bigg)\Bigg(1_{\{W_1> t\}}+1_{\{W_1\leq t<W_2\}}\Bigg)\Bigg]\\
			&+\E_x\Bigg[-e^{\int_{0}^{S_1}\Psi(\xi_u^{x},r+u)du})g(v_{\Psi}(\xi_{S_1}^{x},r+S_1,t-S_1))1_{\{W_1\leq t\}}\Bigg]\\
			&+\E_x\Bigg[-e^{-\int_{0}^{S_1+S_2}\Psi(\xi_{u}^{x},r+u)du}g(v_{\Psi}(\xi_{S_1+S_2}^{x},r+S_1+S_2,t-S_1-S_2))1_{\{W_2\leq t\}}\Bigg]\\
			&+\E_x\Bigg[\Bigg(1-e^{-\int_{0}^{S_1+S_2}\Psi(\xi_{u}^{x},r+u)du}\Bigg)1_{\{W_2\leq t\}}\Bigg]\\
			&+\E_x\Bigg[e^{-\int_{0}^{S_1+S_2}\Psi(\xi_{u}^{x},r+u)du}v_\Psi(\xi_{S_1+S_2}^{x},r+S_1+S_2,t-S_1-S_2)1_{\{W_2\leq t\}}\Bigg]\\
			&=\E_x\Bigg[\Bigg(1-e^{-\int_0^t\Psi(\xi_s,r+s)ds}\Bigg)\sum_{i=1}^2 \Bigg(\mathbf{1}_{\{W_{i-1}\leq t<W_i\}}\Bigg)\\
			&-\sum_{i=1}^2 e^{-\int_0^{W_i}\Psi(\xi_s,r+s)ds}g(v_\Psi(\xi_{W_i},r+W_i,t-W_i))\mathbf{1}_{\{W_i\leq t\}}\Bigg]\\
			&+\E_x\Bigg[\Bigg(1-e^{-\int_{0}^{W_2}\Psi(\xi_{u}^{x},r+u)du}\Bigg)1_{\{W_2\leq t\}}\Bigg]\\
			&+\E_x\Bigg[e^{-\int_{0}^{W_2}\Psi(\xi_{u}^{x},r+u)du}v_\Psi(\xi_{W_2}^{x},r+W_2,t-W_2)1_{\{W_2\leq t\}}\Bigg].
		\end{split}
	\end{equation}
	By an iterative procedure and using that
	$$\E_x\Bigg[e^{-\int_0^{W_n}\Psi(\xi_s,r+s)\,ds}v_\Psi(\xi_{W_n},r+W_n,t-W_n)\mathbf{1}_{\{W_n\leq t\}}
	+\Bigg(1-e^{-\int_0^{W_n}\Psi(\xi_s,r+s)ds}\Bigg)\mathbf{1}_{\{W_n\leq t\}}\Bigg]\leq 2P(W_n\leq t)\to 0
	$$
	as $n\to\infty$ for all $t>0$, we get
	\begin{eqnarray*}\lefteqn{
			v_\Psi(x,r,t)}\\&=&\E_x\left[\left(1-e^{-\int_0^t\Psi(\xi_s,r+s)\,ds}\right)\sum_{i=1}^\infty\mathbf{1}_{\{W_{i-1}\leq t<W_i\}}
		-\sum_{i=1}^\infty e^{-\int_0^{W_i}\Psi(\xi_s,r+s)\,ds}g(v_\Psi(\xi_{W_i},r+W_i,t-W_i))\mathbf{1}_{\{W_i\leq t\}}\right]
		\\&=&\E_x\left[\left(1-e^{-\int_0^t\Psi(\xi_s,r+s)\,ds}\right)
		-\int_{0}^{t} e^{-\int_0^{u}\Psi(\xi_s,r+s)\,ds}g(v_\Psi(\xi_{u},r+u,t-u))dN(u)\right],
	\end{eqnarray*}
	which is equivalent to
	\begin{eqnarray*}
		v_\Psi(x,r,t)&=&\E_x\left[1-e^{-\int_0^t\Psi(\xi_s,r+s)\,ds}-\int_0^te^{-\int_0^{u}\Psi(\xi_s,r+s)\,ds}g(v_\Psi(\xi_{u},r+u,t-u)) \,d U(u)\right],
	\end{eqnarray*}
	where $U(u)=\E_x[N_u]$, $u\geq0$.\hfill$\Box$
	
	\begin{remark} In the case of critical binary branching and exponential lifetimes with rate $V>0$, i.e., $g(s)=\frac{1}{2}s^2$ and $dU(u)=V\,du$, equation (\ref{LaplaceV1}) reduces to
		$$
		v_\Psi(x,r,t)=\nonumber\E_x\left[1-e^{-\int_0^t\Psi(\xi_s,r+s)\,ds}\right]
		-V\int_0^t\E_x\left[e^{-\int_0^{t-u}\Psi(\xi_s,r+s)\,ds}\left(\frac{1}{2}(v_\Psi(\xi_{u},r+t-u,u))^2 \right)\right]du,
		$$
		hence, from the Feynman-Kac formula we get
		\begin{eqnarray*}
			\frac{\partial}{\partial t}v_\Psi(x,r,t)&=&\left(\Delta_\alpha+\frac{\partial}{\partial r}\right)v_\Psi(x,r,t)+\Psi(x,r)(1-v_\Psi(x,r,t))-\frac{V}{2}\LL(v_\Psi(x,r,t)\RR)^2\\v_\Psi(x,r,0)&=&0,
		\end{eqnarray*}
		which is equation (3.20) in \cite{BGT1}.
	\end{remark}
	For any $r\in\R$ we set
	\begin{equation}\label{h_definition}
		f (x,r,t):=\E_x\left[e^{-\int_{0}^{t}\Psi(\xi_{u}^{x},r+u)\,du}\right],\quad x\in\R^d,\quad t\ge0.
	\end{equation}
	It follows from the  Feynman-Kac formula that $f$ solves in mild sense the partial differential equation
	$$
	\frac{\partial f}{\partial t}(x,r,t)
	=\left(\Delta_\alpha+\frac{\partial}{\partial r}\right)f(x,r,t)-\Psi(x,r)f(x,r,t)
	$$
	with initial value  $f(x,r,0)=1$, i.e.
	\begin{eqnarray}\label{h1}
		f(x,r,t)
		&=&1-\int_{0}^{t}\S_{u}\left[\Psi(\cdot,r+u)f(\cdot,r+u,t-u)\right](x)\,du.
	\end{eqnarray}
	We finish this section with the following result, which will be useful to prove convergence of the finite-dimensional distributions in Theorem  \ref{main1} and Theorem \ref{main}.
	\begin{lemma}\label{v_charactarization}Assume that $U$ is  absolutely continuous with density  function  ${\cal U}.$ The function $v_\Psi(x,r,t)$ in (\ref{LaplaceV1}) can be written as
		\begin{eqnarray}
			\lefteqn{v_\Psi(x,r,t)}\nonumber\\ &=&\nonumber\int_{0}^{t}\S_u\left[\Psi(\cdot,r+u)f(\cdot,r+u,t-u)\right](x)du-\int_0^{t}\S_u g(v_\Psi(\cdot,r+u,t-u))(x)\,dU(u)\\
			&&\nonumber +\int_{0}^{t}\int_{0}^{t-z}\S_z\left[\Psi(\cdot,r+z)\E_\cdot\left(e^{-\int_{0}^{u}\Psi(\xi_s^{\cdot},r+z+s)ds}g(v_\Psi(\xi_{u}^{\cdot},r+u+z,t-z-u))\right)\right](x)\,\mathcal{U}(u+z)\,du\,dz.
		\end{eqnarray}
	\end{lemma}
	\proof
	Let us define,  for some fixed $s\in\R_+$ ,
	\begin{equation}\label{k2}
		k(x,r,\sigma):=\E_x\Bigg[e^{-\int_0^{\sigma}\Psi(\xi_{u}^{x},r+u)\,du}g\LL(v_\Psi(\xi_{\sigma}^{x},r+\sigma,s)\RR)\Bigg].
	\end{equation}
	Notice that $k(x,r,\sigma)$ also depends on the fixed parameter $s$ but we omit such dependency.
	Using again the  Feynman-Kac formula we have
	$$
	\frac{\partial }{\partial \sigma }k(x,r,\sigma)
	=\left(\Delta_{\alpha}+\frac{\partial}{\partial r}\right)k(x,r,\sigma)-\Psi(x,r)k(x,r,\sigma),
	$$
	or
	\begin{equation}\label{k3}
		k(x,r,\sigma)=\S_{\sigma}g(v_\Psi(\cdot,r+\sigma,s))(x)-\int_{0}^{\sigma}\S_{\sigma-w}\left[\Psi(\cdot,r+\sigma-w)k(\cdot,r+\sigma-w,w)\right](x)\,dw.
	\end{equation}
	Due to (\ref{h_definition}) and (\ref{k2}), equation (\ref{LaplaceV1}) can be expressed as
	\begin{equation}\label{LaplaceV2}
		v_\Psi(x,r,t)=1-f(x,r,t)-\int_0^{t}k(x,r,t-v)\,\mathcal{U}(t-v)\,dv.
	\end{equation}
	From (\ref{h1}) and (\ref{k3}) we obtain
	\begin{eqnarray}\label{vaux}
		v_\Psi(x,r,t)\nonumber &=&\int_{0}^{t}\S_u\left[\Psi(\cdot,r+u)f(\cdot,r+u,t-u)\right](x)\,du
		-\int_0^{t}\S_{t-v}g(v_\Psi(\cdot,r+t-v,v))(x)\,\mathcal{U}(t-v)\,dv\nonumber\\
		& &+\int_0^t\int_{0}^{t-v}\S_{t-v-w}\left[\Psi(\cdot,r+t-v-w)k(\cdot,r+t-v-w,w)\right](x)\,dw\, \mathcal{U}(t-v)\,dv
	\end{eqnarray}
	with
	\begin{equation}\label{KKK}
		k(x,r+t-v-w,w)=\E_x\left[e^{-\int_{0}^{w}\Psi(\xi_u^{x},r+t-v-w+u)\,du}g(v_\Psi(\xi_{w}^{x},r+t-v-w+w,
		v)\right], \quad x\in\R^d.
	\end{equation}
	Using (\ref{KKK}) and making the change of variables $z=t-v-w$, the double integral in  (\ref{vaux}) transforms into
	$$\int_{0}^{t}\int_{0}^{t-v}\S_{z}\left[\Psi(\cdot,r+z)\E_\cdot\left(e^{-\int_{0}^{t-v-z}\Psi(\xi_u^{\cdot},r+z+u)\,du}g(v_{\Psi}(\xi_{t-v-z}^{\cdot},r+t-v,v)\right)\right](x)\,dz\,\mathcal{U}(t-v)\,dv.
	$$
	Then, firstly applying Tonelli's Theorem and then making the change of variables $u=t-z-v$, in the  double integral in  (\ref{vaux}) we get
	\begin{eqnarray*}
		&& \int_{0}^{t}\int_{0}^{t-v}\S_{z}\left[\Psi(\cdot,r+z)\E_\cdot\left(e^{-\int_{0}^{t-v-z}\Psi(\xi_u^{\cdot},r+z+u)\,du}g(v_{\Psi}(\xi_{t-v-z}^{\cdot},r+t-v,v))\right)\right](x)\,dz\,\mathcal{U}(t-v)\,dv
		\\&=& \int_{0}^{t}\int_{0}^{t-z}\S_{z}\left[\Psi(\cdot,r+z)\E_\cdot\left(e^{-\int_{0}^{t-v-z}\Psi(\xi_u^{\cdot},r+z+u)\,du}g(v_{\Psi}(\xi_{t-v-z}^{\cdot},r+t-v,v))\right)\right](x)\,\mathcal{U}(t-v)\,dv\,dz
		\\&=&\int_{0}^{t}\int_{0}^{t-z}\S_{z}\left[\Psi(\cdot,r+z)\E_\cdot\left(e^{-\int_{0}^{u}\Psi(\xi_s^{\cdot},r+z+s)ds}g(v_\Psi(\xi_{u}^{\cdot},r+u+z,t-z-u))\right)\right](x)\,\mathcal{U}(u+z)\,du\,dz.
	\end{eqnarray*}
	Finally,  plugging the last identity  into (\ref{vaux}) we conclude the proof.\hfill$\Box$

	\section{Proofs of main results}\label{PROOFS}

	As we mentioned in the first section, our proof of Theorem \ref{main1} and Theorem \ref{main} will relay on the space-time random field method developed in \cite{BGR} and applied in \cite{BGT1} to treat the Markovian case. Briefly described, the  space-time random field method consists in the following. Let $\Upsilon>0$.  For every stochastic process  $X\equiv\{X(t),\ t\geq0\}$  with paths in the Skorokhod space $D([0,\Upsilon],\s)$ of c\`adl\`ag functions $\omega:[0,\Upsilon]\to \s$ let $\tilde{X}$ be the random element of ${\cal S}^\prime(\R^{d+1})$   defined by
	\[\langle \tilde{\Phi},\tilde{X}\rangle=\int_0^{\Upsilon}\langle\tilde{\Phi}(\cdot,s),X(s)\rangle\, ds,\quad\tilde{\Phi}\in {\cal S}(\R^{d+1}).\]
	If $X$ is a.s. continuous at $\Upsilon$, then the law of $\tilde{X}$ determines that of $X$. Moreover,  if a family $\{X_T,\ T\geq1\}$ of  $\s$-valued processes with paths in $C([0,\Upsilon],\s)$ is tight, and $\tilde{X}_T$ converges in distribution in $S^\prime(\R^{d+1})$ as $T\rightarrow\infty$, then $X_T\Rightarrow X$ in $C([0,\Upsilon],\s)$ as $T\rightarrow\infty$ for some $\s$-valued process $X$. Without loss of generality, in the sequel  we will assume $\Upsilon =1$.

	\subsection{Proof of theorems \ref{main1} and \ref{main}}
	\subsubsection{Tightness}
	
	We start by proving that  the sequence $\{\J_T,\ T\geq M_{\alpha,d,\gamma}\}$ is tight, for some constant $M_{\alpha,d,\gamma}>0$. Recall that for $0\leq s\leq t$ and $\psi,\varphi\in {\cal S}(\R^d)$,
	\begin{equation}\label{Covar}
		\mbox{Cov}(\l\varphi,Z(s)\r,\l\psi,Z(t)\r)=\langle\varphi\S_{t-s}\psi,\lambda\rangle+\int_0^s\int_{\mathbb{R}^d}(\S_{s-r}\varphi)(x)(\S_{t-r}\psi)(x)\,dx\,dU(r);
	\end{equation}
	see  \cite{LM-MS}. Let
	$\hat{\varphi}$ be the
	Fourier transform
	$
	\hat{\varphi}(x)=\int_{\mathbb{R}^d}e^{ix\cdot y}\varphi(y)\,dy,
	$ $x\in\Rd$,
	where $x\cdot y$ denotes the inner product in $\mathbb{R}^d$.
	Using (\ref{Covar}),  Plancherel's formula 
	and the identity $\widehat{\S_t\varphi}(x)=e^{-t|x|^\alpha}\hat{\varphi}(x)$,  we deduce that
	\begin{equation}\label{CovarFourier}
		\mbox{Cov}\left(\langle\varphi,Z(s)\rangle\langle\psi,Z(t)\rangle\right)=\frac{1}{(2\pi)^d}\int_{\mathbb{R}^d}\hat{\varphi}(y)\overline{\hat{\psi}(y)}\left[e^{-(t-s)|y|^\alpha}+\int_0^se^{-(t+s-2r)|y|^\alpha}
		dU(r)\right]dy.
	\end{equation}
	Due to (\ref{CovarFourier}),
	for any $\psi\in\ss$,
	\begin{equation}\label{OccupVar}
		\E\LL[\l\psi,\J_T(t)\r-\l\psi,\J_T(s)\r\RR]^2=\frac{T^2}{H_T^2}\int_s^t\int_s^t\mbox{Cov}(\l\psi,Z(Tu)\r,\l\psi,Z(Tv)\r)\,du\,dv \ = \ I + II
	\end{equation}
	where
	\[I=2\frac{ T^{d/\alpha-\gamma}}{(2\pi)^d}\int_s^t\int_s^v\int_{\Rd}|\hat{\psi}(y)|^2e^{-T(v-u)|y|^\alpha}\,dy\,du\,dv\]
	and
	\begin{eqnarray*}
		II&=&2\frac{T^{d/\alpha-\gamma}}{(2\pi)^d}\int_s^t\int_s^v\int_{\Rd}|\hat{\psi}(y)|^2\int_0^ue^{-(Tv+Tu-2r)|y|^\alpha}dU(r)\,dy\,du\,dv\\
		&=&2\frac{T^{d/\alpha-\gamma}}{(2\pi)^d}\int_s^t\int_s^v\int_{\Rd}|\hat{\psi}(y)|^2\int_0^ue^{-T(v+u-2r)|y|^\alpha}dU(Tr)\,dy\,du\,dv.
	\end{eqnarray*}
	We first deal with the term $I$. For any  $s,t\in[0,1]$ with $s\leq t$,
	\begin{equation*}
		\begin{split}
			\int_{s}^{t}\int_{s}^{v}e^{-T(v-u)|y|^{\alpha}}\,du\,dv&=
			\frac{1}{T|y|^{\alpha}}\int_{s}^{t}(1-e^{-T|y|^{\alpha}(v-s)})\,dv \ = \ \frac{1}{T|y|^{\alpha}}\int_{0}^{t-s}(1-e^{-Tv|y|^{\alpha}})\,dv\\
			&\leq \frac{1}{T|y|^{\alpha}}\int_{0}^{t-s}(T|y|^{\alpha}v)^{\delta}\,dv \
			= \ \frac{T^{\delta-1}}{\delta+1}\frac{1}{|y|^{\alpha(1-\delta)}}(t-s)^{\delta+1},
		\end{split}
	\end{equation*}
	where the inequality above follows from the relation $1-e^{-x}\leq x^{\delta}$, valid for $x>0$ and  $0<\delta\leq 1$. Since by assumption  $\alpha\gamma<d<\alpha(1+\gamma)$,  choosing  $\delta=1+\gamma -d/\alpha$ we get $\delta\in(0,1]$ and
	\begin{equation}\label{I-stimation}
		I\leq \frac{2}{(2\pi)^d h}\int_{\R^d}\frac{|\hat{\psi}(y)|^2}{|y|^{d-\alpha\gamma}}\, dy \times (t-s)^{h},\quad \mbox{with $h=2+\gamma-d/\alpha$,}
	\end{equation}
	and the last integral is finite because $d>\alpha\gamma$ and $\psi\in\mathcal{S}(\R^d)$.

	\begin{remark}\label{remark_revision}
		Notice that assumption (\ref{tail}) implies the equivalence
		$U(T)\sim T^{\gamma}/\Gamma(1+\gamma)$ as $T\to\infty$ (see e.g. \cite[(8.6.3)]{B.et.al.}  or \cite[Thm. 2.2.2]{Anderson}), which in turn entails
		$$
		\frac{U(Tr)}{U(T)}\to r^{\gamma}\quad\mbox{as}\quad T\to\infty\quad
		\mbox{for any $r\in[0,1]$.}
		$$
		Moreover, $\int_{0}^{1}\frac{dU(Ts)}{U(T)}=1$ and for all $0\leq a< b\leq 1$,
		$$\int_{a}^{b}\frac{dU(Ts)}{U(T)}=\frac{U(Tb)-U(Ta)}{U(T)}\to b^{\gamma}-a^{\gamma}=\int_{a}^{b}\gamma s^{\gamma-1}\,ds\quad\mbox{as $T\to\infty$},$$
		hence the measure  $\hat{U}_T$ defined on  $([0,1],\mathcal{B}([0,1]))$  by
		\begin{equation}\label{Renewallimit}
			\hat{U}_{T}([0,u])=\int_{0}^{u}\frac{dU(Ts)}{U(T)}
		\end{equation}
		weakly converges to the measure on  $([0,1],\mathcal{B}([0,1]))$ having density function $\gamma s ^{\gamma-1}\mathbf{1}_{(0,1)}(s).$
	\end{remark}
	Before working the second term $II$, we prepare with a lemma.
	
	\begin{lemma}\label{lemma_1_revision}
		Let $s,t\in[0,1]$ with $s\le t$.
		There exist constants $C>0$ and $M_{\alpha,d,\gamma}>0$ such that for all $T>M_{\alpha,d,\gamma}$,
		$$
		\int_s^t\int_s^v\int_0^u
		(v+u-2r)^{-d/\alpha}\,d\frac{U(Tr)}{U(T)}\,du\,dv
		\le
		C
		\int_s^t\int_s^v\int_0^u(v+u-2r)^{-d/\alpha}r^{\gamma-1}dr\,du\,dv.
		$$
	\end{lemma}
	\begin{proof} We give the proof  for the case  $d\neq\alpha$ only; the case $d=\alpha$ can be worked in a similar way. We have
		\begin{equation*}
			\begin{split}
				&\int_s^t\int_s^v\int_0^u (v+u-2r)^{-d/\alpha} d\frac{U(Tr)}{U(T)}du\,dv=\int_{s}^{t}\int_{r}^{t}\int_r^v(v+u-2r)^{-\frac{d}{\alpha}} \, du \, dv \, d\frac{U(Tr)}{U(T)}\\
				&+\int_{0}^{s}\int_{s}^{t}\int_s^v (v+u-2r)^{-\frac{d}{\alpha}} \, du \, dv \, d\frac{U(Tr)}{U(T)}\\
				&=\int_{s}^{t}\int_{r}^{t}\frac{(2^{1-\frac{d}{\alpha}}-1)}{1-\frac{d}{\alpha}}(v-r)^{1-\frac{d}{\alpha}} \, dv \, d\frac{U(Tr)}{U(T)}+\int_{0}^{s}\int_{s}^{t}\Bigg[ \frac{2^{1-\frac{d}{\alpha}}(v-r)^{1-\frac{d}{\alpha}}-(v+s-2r)^{1-\frac{d}{\alpha}}}{1-\frac{d}{\alpha}}\Bigg] \, dv \, d\frac{U(Tr)}{U(T)}\\
				&=\int_0^t f_{t,s}(r)\,d\frac{U(Tr)}{U(T)},
			\end{split}
		\end{equation*}
		with
		\begin{eqnarray*}
			f_{t,s}(r)&:=&\frac{(2^{1-\frac{d}{\alpha}}-1)(t-r)^{2-\frac{d}{\alpha}}}{(2-\frac{d}{\alpha})(1-\frac{d}{\alpha})}\mathbf{1}_{\{s\leq r\leq t\}} \\
			&&
			+\frac{2^{1-\frac{d}{\alpha}}\left((s-r)^{2-\frac{d}{\alpha}}+(t-r)^{2-\frac{d}{\alpha}}-2^{\frac{d}{\alpha}-1}(t+s-2r)^{2-\frac{d}{\alpha}}\right)}{(2-\frac{d}{\alpha})(1-\frac{d}{\alpha})}\mathbf{1}_{\{0\leq r<s\}}.
		\end{eqnarray*}
		Due to the term $(t+s-2r)^{2-\frac{d}{\alpha}}$, the function $f_{t,s}$ is supported in the interval $[0,2]$. Since the function
		$x\mapsto x^{2-\frac{d}{\alpha}}$ is bounded and uniformly continuous over the interval $[0,2]$, for any $\epsilon>0$ there exists $\delta>0$ such that given $x,y\in [0,2]$,
		\[ \left|x^{2-\frac{d}{\alpha}}-y^{2-\frac{d}{\alpha}}\right|<\epsilon\quad \mbox{whenever}\quad |x-y|<\delta.\]
		Now, given $\delta>0$ there exist $k \in \mathbb{N}$ and $x_0,x_1,...,x_k \in \mathbb{R}_{+}$ such that $x_0=0<x_1<\cdots<x_k=2$ and $|x_i-x_{i-1}|<\frac{\delta}{2}$ for  $i=1,2,...,k$. We set  $g_{t,s}(r):=\sum_{i=1}^{k}f_{t,s}(x_{i-1})1_{[x_{i-1},x_i)}(r)$. Hence, for any $r\in[0,2]$,
		$$
		|f_{t,s}(r)-g_{t,s}(r)|\le \frac{\left|(2^{1-\frac{d}{\alpha}}-1)\right|\epsilon}{\left|(2-\frac{d}{\alpha})(1-\frac{d}{\alpha})\right|}\mathbf{1}_{\{s\leq r\leq t\}} + 
		\frac{2^{1-\frac{d}{\alpha}}\left(2\epsilon+2^{\frac{d}{\alpha}-1}\epsilon\right)}{\left|(2-\frac{d}{\alpha})(1-\frac{d}{\alpha}\right|)}\mathbf{1}_{\{0\leq r<s\}}\le M_{\alpha,d}\epsilon,
		$$
		for some positive constant $M_{\alpha,d}$ depending only on $\alpha$ and  $d$. Moreover, we  choose $M_{\alpha,d}$ so that  $|f_{t,s}(r)|\le M_{\alpha,d}$ for $r\in[0,2]$.
		Therefore,
		\begin{eqnarray} \nonumber
				& &\Bigg|\int_{0}^{t}f_{t,s}(r)\frac{dU(Tr)}{U(T)}-\int_{0}^{t}f_{t,s}(r)\gamma r^{\gamma-1}dr\Bigg|\leq \Bigg|\int_{0}^{t}f_{t,s}(r)\frac{dU(Tr)}{U(T)}-\int_{0}^{t}g_{t,s}(r)\frac{dU(Tr)}{U(T)}\Bigg|\\ \nonumber
				&  &+\Bigg|\int_{0}^{t}g_{t,s}(r)\frac{dU(Tr)}{U(T)}-\int_{0}^{t}g_{t,s}(r)\gamma r^{\gamma-1}dr\Bigg|+\Bigg|\int_{0}^{t}g_{t,s}(r)\gamma r^{\gamma-1}dr-\int_{0}^{t}f_{t,s}(r)\gamma r^{\gamma-1}dr\Bigg|\\ \label{ExTrA}
				& &\leq \epsilon M_{\alpha,d} \frac{U(Tt)}{U(T)}+M_{\alpha,d}\sum_{i=1}^{k}\Bigg|\frac{U(Tx_i)-U(Tx_{i-1})}{U(T)}-(x_i^{\gamma}-x_{i-1}^{\gamma})\Bigg|1_{\{x_i\leq t\}}+\epsilon M_{\alpha,d}t^{\gamma}
		\end{eqnarray}
		On the other hand, since $\frac{U(Tr)}{U(T)}\to r^{\gamma}$ uniformly on  $[0,1]$, there  exists $M_{\alpha,d,\gamma}>0$ such that for $T>M_{\alpha,d,\gamma}$:
		$$\Bigg|\frac{U(Tx_i)-U(Tx_{i-1})}{U(T)}-(x_i^{\gamma}-x_{i-1}^{\gamma})\Bigg|1_{\{x_i\leq t\}}<\frac{\epsilon}{M_{\alpha,d}k}\text{ for all }i=1,2,...,k.$$
		Plugging this inequality into (\ref{ExTrA}) yields the result. \ter
	\end{proof}
	
	To bound from above  in a useful way the second term $II$  we proceed as follows.
	Given  $s,t\in[0,1]$ with $s\leq t$, since $\hat{\psi}$ is bounded we have
	\begin{eqnarray*}
		II &\leq& C_\psi \frac{2T^{d/\alpha-\gamma}}{(2\pi)^d}\int_s^t\int_s^v\int_{\Rd}\int_0^ue^{-T(v+u-2r)|y|^\alpha}dU(Tr)\,dy\,du\,dv
		\\&=& C_{\psi}\frac{2T^{d/\alpha-\gamma}}{(2\pi)^d}\int_{\Rd}e^{-|y|^\alpha}\, dy \int_s^t\int_s^v\int_0^u\frac{(v+u-2r)^{-d/\alpha}}{T^{d/\alpha}}\,dU(Tr)\,du\,dv
		\\&=& C_{\psi}\frac{2}{(2\pi)^d}\frac{U(T)}{T^{\gamma}}\int_{\Rd}e^{-|y|^\alpha} dy \int_s^t\int_s^v\int_0^u (v+u-2r)^{-d/\alpha} d\left[\frac{U(Tr)}{U(T)}\right]du\,dv,
	\end{eqnarray*}
	where condition (\ref{tail}) implies that $U(T)\sim T^{\gamma}/\Gamma(1+\gamma)$ as $T\to\infty$.
	Therefore, from  Lemma \ref{lemma_1_revision} and taking $T\ge M_{\alpha, d, \gamma}$ bigger if necessary, we deduce that 
	\begin{equation}\label{IIfirst-estimation}
		II \ \leq \  C(\psi,\alpha,\gamma) \frac{2}{(2\pi)^d}\int_s^t\int_s^v\int_0^u (v+u-2r)^{-d/\alpha} r^{\gamma-1}
		\, dr \, du\,dv
	\end{equation}
	for some constant $C(\psi,\alpha,\gamma)>0.$
	For $u<v$, we have  \begin{eqnarray}\label{TightnessIntegral}
		\int_0^u (v+u-2r)^{-d/\alpha} r^{\gamma-1}
		\, dr &=&2^{-\gamma}\int_{0}^{2u}(u+v-r)^{-\frac{d}{\alpha}}r^{\gamma-1}dr\nonumber
		\\&=&2^{-\gamma}(u+v)^{\gamma-\frac{d}{\alpha}}\int_{0}^{\frac{2u}{u+v}}(1-r)^{-\frac{d}{\alpha}}r^{\gamma-1}dr.
	\end{eqnarray}
	To deal with the last integral we  work separately the two cases $\alpha\gamma<d<\alpha$ and $\alpha\leq d<\alpha(1+\gamma)$.
	
	\noindent{\bf Case $\alpha\gamma<d<\alpha$.} For the integral that appears in (\ref{TightnessIntegral}),
	$$
	\int_{0}^{\frac{2u}{u+v}}(1-r)^{-\frac{d}{\alpha}}r^{\gamma-1}\,dr\leq \int_{0}^{1}(1-r)^{-\frac{d}{\alpha}}r^{\gamma-1}\,dr={\cal B}(1-d/\alpha,\gamma)\equiv C<\infty,
	$$
	where ${\cal B}(p,q)$ denotes the beta function. It follows that
	\begin{eqnarray*}
		\int_s^t\int_s^v\int_0^u (v+u-2r)^{-d/\alpha} r^{\gamma-1}
		\, dr \, du\,dv&<&2^{-\gamma}C\int_s^t\int_s^v(v+u)^{\gamma-d/\alpha} du\,dv
		\\&=&2^{-\gamma}C\int_s^t \left[(2v)^{1+\gamma-d/\alpha}-(v+s)^{1+\gamma-d/\alpha}\right] dv.
	\end{eqnarray*}
	Since condition $\alpha\gamma<d<\alpha$ implies  $0<1+\gamma-d/\alpha<1$, using Hölder continuity we get
	$$
	\int_s^t \left[(2v)^{1+\gamma-d/\alpha}-(v+s)^{1+\gamma-d/\alpha}\right] \,dv \leq  C_1\int_s^t (v-s)^{1+\gamma-d/\alpha} \, dv
	=\frac{C_1}{2+\gamma-d/\alpha}(t-s)^{2+\gamma-d/\alpha}.
	$$
	We conclude that for sufficiently large $T$,
	\begin{equation}\label{lowerdimensions}
		II \ < \  C(\psi,d,\alpha,\gamma)(t-s)^{h} \quad\mbox{with $h=2+\gamma-d/\alpha.$
		}
	\end{equation}
	
	\noindent{\bf Case $\alpha\leq d<\alpha(1+\gamma)$.} Notice that
	
	\begin{eqnarray*}
		\int_{0}^{\frac{2u}{u+v}}(1-r)^{-\frac{d}{\alpha}}r^{\gamma-1}\,dr\begin{cases}\leq\displaystyle \int_{0}^{\frac{1}{2}}(1-r)^{-\frac{d}{\alpha}}r^{\gamma-1}\,dr,& \mbox{if }\frac{2u}{v+u}<\frac{1}{2},\\
			$\,$ \\
			= \displaystyle\int_{0}^{\frac{1}{2}}(1-r)^{-\frac{d}{\alpha}}r^{\gamma-1}\,dr+ \int_{\frac{1}{2}}^{\frac{2u}{u+v}}(1-r)^{-\frac{d}{\alpha}}r^{\gamma-1}\,dr,& \mbox{if }\frac{2u}{v+u}\geq \frac{1}{2}.
		\end{cases}
	\end{eqnarray*}
	Now, since $\gamma\in(0,1)$ we have that $ \int_{0}^{\frac{1}{2}}(1-r)^{-\frac{d}{\alpha}}r^{\gamma-1}\,dr<\infty$. For the case  $\frac{2u}{v+u}\geq \frac{1}{2}$ we notice that
	$r^{\gamma-1}\leq 2^{-(\gamma-1)}$ for all $r\in\left[ \frac{1}{2},\frac{2u}{v+u}\right].$
	Thus, if $\alpha<d$ we obtain
	\begin{eqnarray*}
		\int_{\frac{1}{2}}^{\frac{2u}{u+v}}(1-r)^{-\frac{d}{\alpha}}r^{\gamma-1}\,dr&\leq &2^{-(\gamma-1)}\int_{\frac{1}{2}}^{\frac{2u}{u+v}}(1-r)^{-\frac{d}{\alpha}}\,dr
		\leq\frac{2^{-(\gamma-1)}}{d/\alpha-1}\left(\frac{v-u}{v+u}\right)^{1-d/\alpha}.
	\end{eqnarray*}
	In fact, in the case we are dealing with, we have that $\gamma-d/\alpha<1-d/\alpha<0$. Therefore,
	\begin{eqnarray*}
		\int_{\frac{1}{2}}^{\frac{2u}{u+v}}(1-r)^{-\frac{d}{\alpha}}r^{\gamma-1}\,dr\leq \frac{2^{-(\gamma-1)}}{d/\alpha-1}\left(\frac{v-u}{v+u}\right)^{\gamma-d/\alpha}.
	\end{eqnarray*}
	If $d=\alpha$,
	\begin{eqnarray*}
		\int_{\frac{1}{2}}^{\frac{2u}{u+v}}(1-r)^{-\frac{d}{\alpha}}r^{\gamma-1}\,dr&\leq &\left(\frac{1}{2}\right)^{\gamma-1}\int_{\frac{1}{2}}^{\frac{2u}{u+v}}(1-r)^{-1}\,dr
		\leq-\left(\frac{1}{2}\right)^{\gamma-1}\ln\left(\frac{v-u}{v+u}\right)
		\\&\leq&C\left(\frac{1}{2}\right)^{\gamma-1}\left(\frac{v-u}{v+u}\right)^{\gamma-d/\alpha},
	\end{eqnarray*}
	for some constant $C>0$, where the last  equality follows from the boundedness of the function $x\mapsto -(\ln x)/x^{\gamma-1}$  on the interval  $[0,\frac{1}{2}]$. Therefore, from (\ref{IIfirst-estimation}), (\ref{TightnessIntegral}) and the last estimates we get that for $T$ large enough, and for some positive constant $C$,
	\begin{equation}\label{upperdimensions}
		II \leq  C\int_s^t\int_s^v (v-u)^{\gamma-d/\alpha} \,du\, dv\leq \frac{C}{h(h-1)}(t-s)^h,
	\end{equation}
	where $h=2+\gamma-d/\alpha$. \hfill$\Box$

	We are now ready  to state and prove the following
	\begin{proposition} \label{ThightnessProp}Let  $\alpha\gamma<d<\alpha(1+\gamma)$ and $H_T^2=T^{2+\gamma-d/\alpha}$. There exists a constant $M_{\alpha,d,\gamma}>0$ such that the sequence of processes  $\{\J_T,\ T\geq {M_{\alpha,d,\gamma}}\}$  is tight, where $\J_T$ is defined in (\ref{rescaled}).
	\end{proposition}
	\proof From (\ref{OccupVar}), (\ref{I-stimation}), (\ref{lowerdimensions}) and (\ref{upperdimensions}) it follows that, for $T$ large enough,
	\begin{equation}
		\E\LL[\l\psi,\J_T(t)\r-\ph\l\psi,\J_T(s)\r\RR]^2\leq C|t-s|^{\rho},\quad s,t\geq0,
	\end{equation}
	where $\rho=2+\gamma-d/\alpha>1$ because $1+\gamma-d/\alpha>0$ due to the assumption $d<\alpha(1+\gamma)$. From  \cite[Theorem 13.5]{Billingsley}  we get that for each $\psi\in\ss$ the sequence of processes $\{\l\psi,\J_T(t)\r,T\geq {M_{\alpha,d,\gamma}}\}$ is tight for some ${M_{\alpha,d,\gamma}}$ sufficiently large. Using Mitoma's theorem  \cite[Theorem 3.1]{Mitoma} we get the tightness of $\{\J_T,\ T\geq{M_{\alpha,d,\gamma}}\}$.\hfill$\Box$
	
	\subsubsection{Space-time method: convergence to a Gaussian process}
	
	From (\ref{rescaled}) we deduce that the space-time random field associated to $\{\J_T,\ T\geq1\}$ is given by
	\begin{equation*}
		\l\tilde{\Phi},\tilde{\J}_T\r:=\frac{T}{H_T}\left(\int_0^1\l\Psi(\cdot,s),Z(Ts)\r\, ds-\LL\l\int_0^1\Psi(\cdot,s)\,ds,\Lambda\RR\r\right),\quad\tilde{\Phi}\in S(\R^{d+1}),
	\end{equation*}
	with
	$\Psi(x,s)=\int_s^1\tilde{\Phi}(x,t)\,dt.$
	Since the initial population is a Poisson random field with intensity the Lebesgue measure  $\Lambda$,
	\begin{eqnarray}\label{LaplaceOccupationTime}
		\E\left[e^{-\l\tilde{\Phi},\tilde{\J}_T\r}\right]&=&\nonumber\exp\left\{\int_{\R^d}\int_0^T\Psi_T(x,s)\,ds \,dx+\int_{\R^d}\E_x\left(e^{-\int_0^T\l\Psi_T(\cdot,s),Z(s)\r\, ds} -1\right)\,dx\right\}
		\\&=&\exp\left\{\int_{\R^d}\int_0^T\Psi_T(x,s)\,ds\, dx-\int_{\R^d}v_{\Psi_T}(x,0,T)\,dx\right\},
	\end{eqnarray}
	where $v_{\Psi_T}(x,0,T)$ is given in (\ref{LaplaceV}) and
	$
	\Psi_T(x,s)=\frac{1}{H_T}\Psi(x,\frac{s}{T})
	= \frac{1}{H_T}\int^1_{\frac{s}{T}}\tilde{\Phi}(x,t)\,dt
	$
	with $H_T=T^{(2+\gamma-d/\alpha)/2}$.
	\begin{proposition}\label{fdd_convergence_prop} Let $\tilde{\Phi}$ be of the form $\tilde{\Phi}(x,t)=\phi_{1}(x)\phi_{2}(t)$, where   $\phi_1\in \mathcal{S}(\R^{d})_+$ and $\phi_2\in \mathcal{S}(\mathbb{R})_+$. If $\alpha\gamma<d<\alpha(1+\gamma)$, then
		\begin{eqnarray}\nonumber
			\lefteqn{\lim_{T\to \infty}\E\left[e^{-\langle\tilde{\Phi},\tilde{\J}_T\rangle}\right]}\\&&= \label{fdd_convergence_eq}
			\exp\left\{\frac{\gamma\langle\phi_1,\lambda\rangle^{2}}{\Gamma(1+\gamma)(2\pi)^{d}}\LL(\int_{\mathbb{R}^{d}}e^{-|z|^{\alpha}}\,dz\RR)\int_{0}^{1}\int_{0}^{v}\int_{0}^{u}(u+v-2r)^{-{d}/{\alpha}}r^{\gamma-1}\,dr\,\chi(u)\,\chi(v)\,du\,dv\right\},
		\end{eqnarray}
		where $\chi(\cdot)=\int_{\cdot}^{1}\phi_2(s)\,ds$.
	\end{proposition}
	\proof The proof will be divided into four steps.
	Using Lemma \ref{v_charactarization}, (\Ref{h1}) and recalling that $g(s)=\frac{s^{2}}{2}$, we have
	\begin{eqnarray*}
		\lefteqn{\int_{\R^d}\int_0^T\Psi_T(x,s)\,ds\, dx-\int_{\R^d}v_{\Psi_T}(x,0,T)\,dx}\\
		&=&\int_{\R^d}\int_0^T\Psi_T(x,s)\,ds\, dx-\int_{\R^d}\int_0^T\S_u\left(\Psi_T(\cdot,u)f(\cdot,u,T-u)\right)(x)\,du\,dx\\&&+\int_{\R^d}\int_0^T\S_u\left(g(v_{\Psi_T}(\cdot,u,T-u)\right)(x)\,dU(u)\,dx
		\\&&-\int_{\R^d}\int_0^T\int_0^{T-z}\!\! \S_z\left[\Psi_T(\cdot,z)\E_\cdot\left(e^{-\int_{0}^{u}\Psi_T(\xi_s^{\cdot},z+s)\,ds}g(v_{\Psi_T}(\xi_{u}^{\cdot},u+z,T-z-u))\right)\right](x)\mathcal{U}(u+z)\,du\,dz\,dx
		\\&=&\int_{\R^d}\int_0^T\Psi_T(x,s)\,ds \,dx-\int_0^T \int_{\R^d}\Psi_T(x,u)f(x,u,T-u)\,dx\, du +
		\int_0^T\int_{\R^d} g(v_{\Psi_T}(x,u,T-u)\,dx \,dU(u)
		\\&&
		-\int_{\R^d}\int_0^T\int_0^{T-z}\!\! \S_z\left[\Psi_T(\cdot,z)\E_\cdot\left(e^{-\int_{0}^{u}\Psi_T(\xi_s^{\cdot},z+s)\,ds}g(v_{\Psi_T}(\xi_{u}^{\cdot},u+z,T-z-u))\right)\right](x)\mathcal{U}(u+z)\,du\,dz\,dx,
	\end{eqnarray*}
	where to get the second equality we have used that the Lebesgue measure is invariant for the $\alpha$-stable semigroup. Thus, we write
	\[\int_{\R^d}\int_0^T\Psi_T(x,s)\,ds \,dx-\int_{\R^d}v_{\Psi_T}(x,0,T)\,dx\equiv I_1(T)+I_2(T)+I_3(T)+I_4(T),\]
	where
	\begin{eqnarray}\label{I1}
		I_1(T):=\int_{\R^d}\int_0^T\Psi_T(x,s)\,ds \,dx-\int_0^T \int_{\R^d}\Psi_T(x,u)h(x,u,T-u)\,dx\, du,
	\end{eqnarray}
	\begin{eqnarray}\label{I2}
		I_2(T):=\int_0^T\int_{\R^d} \left[g(v_{\Psi_T}(x,T-s,s))-\frac{1}{2}\left(\int_0^{s}\S_u\Psi_T(\cdot,T+u-s)(x)\,du\right)^{2}\right]\mathcal{U}(T-s)\,dx\,ds,
	\end{eqnarray}
	\begin{eqnarray}\label{I3}
		I_3(T):=\frac{1}{2}\int_{0}^{T}\int_{\mathbb{R}^{d}}\left(\int_0^{s}\S_u\Psi_T(\cdot,T+u-s)(x)\,du\right)^{2}\mathcal{U}(T-s)\,dx\,ds,
	\end{eqnarray}
	\begin{eqnarray}\label{I4}
		\lefteqn{I_4(T):=}\\&&\nonumber-\int_{\R^d}\int_0^T\int_0^{T-r} \S_r\left[\Psi_T(\cdot,r)\E_\cdot\left(e^{-\int_{0}^{u}\Psi_T(\xi_s^{\cdot},r+s)ds}g(v_{\Psi_T}(\xi_{u}^{\cdot},u+r,T-r-u))\right)\right](x)\mathcal{U}(u+r)\,du\,dr\,dx.
	\end{eqnarray}
	We are going to show that
	\begin{equation}\label{limitstocero}
		\lim_{T\rightarrow\infty}I_i(T)=0,\quad i\in\{1,2,4\},
	\end{equation}
	and
	\begin{equation}\label{theconvergence}
		\lim_{T\rightarrow\infty}I_3(T)=\frac{\gamma\langle\lambda,\phi_1\rangle^{2}}{\Gamma(1+\gamma)(2\pi)^{d}}\LL(\int_{\mathbb{R}^{d}}e^{-|z|^{\alpha}}\,dz\RR)\int_{0}^{1}\int_{0}^{v}\int_{0}^{u}(u+v-2r)^{-{d}/{\alpha}}r^{\gamma-1}\,dr\,\chi(u)\,\chi(v)\,du\,dv;
	\end{equation}
	here and below we set  $\chi(t):=\int_{t}^{1}\phi_2(s)\,ds$ and $\chi_T(t):=\chi(\frac{t}{T})$. 
	
	\noindent{\bf Step 1. Proof of (\ref{theconvergence}).} Performing suitable changes of variables,  (\ref{I3}) can be re-written as
	\begin{eqnarray}\label{I3-1}
		2I_3(T)&=&\nonumber\int_{0}^{T}\int_{\mathbb{R}^{d}}\left(\int_{s}^{T}\S_{u-s}\Psi_{T}(\cdot,u)(x)\,du\right)^{2}\mathcal{U}(s) \,dx\, ds
		\\&=& \nonumber\int_{0}^{T}\int_{\mathbb{R}^{d}}\int_s^T\int_s^T\S_{u-s}\Psi_{T}(\cdot,u)(x)\S_{v-s}\Psi_{T}(\cdot,v)(x)\, du \,dv\,\mathcal{U}(s)\, dx \,ds
		\\&=&\frac{1}{H_T^2} \int_{0}^{T}\int_s^T\int_s^T\int_{\mathbb{R}^{d}}\S_{u-s}\phi_1(\cdot)(x)\S_{v-s}\phi_{1}(\cdot)(x)\, dx\,{\chi_T(u)\chi_T(v)}\, du\, dv\,\mathcal{U}(s)\,  ds.
	\end{eqnarray}
	From Plancherel's formula   we have
	\[\int_{\mathbb{R}^{d}}\S_{u-s}\phi_1(\cdot)(x)\S_{v-s}\phi_{1}(\cdot)(x) dx=\frac{1}{(2\pi)^d}\int_{\Rd}\widehat{\S_{u-s}\phi_1}(x)\overline{\widehat{\S_{v-s}\phi_1}}(x)dx.\]
	Moreover, $\widehat{\S_{t}\phi_1}(x)=e^{-t|x|^\alpha}\overline{\hat{\phi_1}(x)}$. Thus, from (\ref{I3-1}) we  get
	\begin{eqnarray}
		2I_3(T)&=& \nonumber\frac{1}{(2\pi)^dH_T^2}\int_{0}^{T}\int_s^T\int_s^T\int_{\mathbb{R}^{d}} e^{-(u-s)|x|^\alpha}\overline{\hat{\phi_1}(x)}e^{-(v-s)|x|^\alpha}{\hat{\phi_1}(x)}\chi_T(u)\chi_T(v) \, dx \,du \,dv\,\mathcal{U}(s) \, ds\\
		&=&\nonumber\frac{1}{(2\pi)^dH_T^2}\int_{0}^{T}\LL[\int_s^T\int_s^T\int_{\mathbb{R}^{d}} e^{-(u+v-2s)|x|^\alpha}\left|\hat{\phi_1}(x)\right|^2 \chi_T(u)\chi_T(v) \,dx \,du\, dv\RR]\,\mathcal{U}(s) \, ds.
	\end{eqnarray}
	After the change of variables $s=Tr$ we obtain
	\begin{eqnarray*}
		2I_3(T)&=& \nonumber\frac{T}{(2\pi)^dH_T^2}\int_{0}^{1}\int_{Tr}^T\int_{Tr}^T\int_{\mathbb{R}^{d}} e^{-(u+v-2Tr)|x|^\alpha}\left|\hat{\phi_1}(x)\right|^2 \chi_T(u)\chi_T(v) \, dx \, du \, dv\,\mathcal{U}(Tr) \, dr
		\\&=&\frac{T^3}{(2\pi)^dH_T^2}\int_{0}^{1}\int_r^1\int_r^1\int_{\mathbb{R}^{d}} e^{-(u+v-2r)T|x|^\alpha}\left|\hat{\phi_1}(x)\right|^2 \chi(u)\chi(v)\, dx\, du\, dv\,\mathcal{U}(Tr) \,dr,
	\end{eqnarray*}
	where to get the last identity we again changed variables.
	Performing further the change of variables $z=((u+v-2r)T)^{1/\alpha}x,$ the last expression is equivalent to
	\begin{eqnarray*}
		I_3(T)&=&\frac{T^{3-d/\alpha}}{(2\pi)^dH_T^2}\int_{0}^{1}\int_r^1\int_r^1\int_{\mathbb{R}^{d}} e^{-|z|^\alpha}\left|\hat{\phi_1}\left(((u+v-2r)T)^{-1/\alpha}z\right)\right|^2 \\&&\times(u+v-2r)^{-d/\alpha}\chi(u)\chi(v) \,  dz\,  du\,  dv\, \mathcal{U}(Tr) \, dr.
	\end{eqnarray*}
	Since $H_T^2=T^{2+\gamma-d/\alpha}$, we obtain
	\begin{eqnarray*}
		I_3(T)&=&\frac{1}{(2\pi)^d}\int_{0}^{1}\int_r^1\int_r^1\int_{\mathbb{R}^{d}} e^{-|z|^\alpha}\left|\hat{\phi_1}\left(((u+v-2r)T)^{-1/\alpha}z\right)\right|^2 \\&&\times(u+v-2r)^{-d/\alpha}\chi(u)\chi(v) \,  dz\,  du\,  dv\, d \left[\frac{U(Tr)}{T^\gamma} \right].
	\end{eqnarray*}
	Hence, changing the order of integration, by the weak convergence of  (\ref{Renewallimit}) and uniform convergence of the integrand we get  $$
	\lim_{T\rightarrow\infty }I_3(T)=\frac{|\hat{\phi_1}(0)|^2}{(2\pi)^d}\frac{\gamma}{\Gamma(1+\gamma)}\left(\int_{\mathbb{R}^{d}} e^{-|z|^\alpha}dz\right)\int_{0}^{1}\int_0^v\int_0^u (u+v-2r)^{-d/\alpha}r^{\gamma-1} \, dr\,\chi(u)\,\chi(v) \,  du\,  dv.
	$$  
	\ter
	
	\noindent{\bf Step 2. Proof of (\ref{limitstocero}) for $i=1$.}
	Using equation (\ref{h1}) we have that
	\begin{eqnarray}\label{I1.3}
		I_1(T)=\int_{\Rd}\int_0^T\Psi_T(x,u)\int_0^{T-u} \S_s\left(\Psi_T(\cdot,u+s)f(\cdot,u+s,T-u-s)\right)(x)\,ds\, du\, dx.
	\end{eqnarray}
	Notice that $f\leq 1$ because by assumption $\Psi$ is nonnegative (see (\ref{h_definition})). Therefore, letting $c>0$ be an upper bound for $\phi_2$,
	
	\begin{eqnarray*}
		I_1(T)&\leq&\frac{c}{H_T^2}\int_{\Rd}\int_0^T \int_0^{T} \phi_1(x)\left(\S_s\phi_1\right)(x) \, ds\, du\, dx
		=\frac{cT}{H_T^2} \int_{\Rd}\left|\hat{\phi}_1(x)\right|^{2}\int_{0}^{T}e^{-s|x|^{\alpha}}\, ds\, dx.
	\end{eqnarray*}
	Then, clearly, for any $\delta \in [0,1]$ we have
	\begin{equation}\label{revision-w}\int_{0}^{T}e^{-s|x|^{\alpha}}ds\leq T\wedge \left(\frac{1}{|x|^{\alpha}}\right)\leq T^{1-\delta}\frac{1}{|x|^{\alpha \delta}}.\end{equation}
	If $d>\alpha$ we use (\ref{revision-w}) with $\delta=1$. If $\gamma \alpha<d\leq \alpha$ we set $\delta=\frac{d}{\alpha}-\frac{\gamma}{2}$ to obtain that $I_1(T)$ converges to $0$. \ter

	\noindent{\bf Step 3. Proof of (\ref{limitstocero}) for $i=4$.}  By performing the change of variables $v=T-r-u$ in (\ref{I4})
	we obtain
	\begin{eqnarray*}
		\lefteqn{-I_4(T)=}\\&&\int_{\R^d}\int_0^T\int_0^{T-r} \S_r\left[\Psi_T(\cdot,r)\E_\cdot\left(e^{-\int_{0}^{T-r-v}\Psi_T(\xi_s,r+s)\,ds}g(v_{\Psi_T}(\xi_{T-r-v}^{\cdot},T-v,v))\right)\right](x)\,\mathcal{U}(T-v)\,dv\,dr\,dx,
	\end{eqnarray*}
	where, due to (\ref{LaplaceV2}) and the fact that $k\ge0$,
	\begin{eqnarray*}
		v_\Psi(x,r,s)& = &
		1-f(x,r,s) -\int_0^sk(x,r,s-v)\,{\cal U}(s-v)\,dv\ \le \ 1-f(x,r,s)\\
		&=& \int_0^s\S_u\LL[\Psi(\cdot,r+s)f(\cdot,r+s,s-u)\RR](x)\,du
		\leq \int_0^s \S_{s-u}\Psi(\cdot,r+s-u)(x)\,du
	\end{eqnarray*}
	since
	$|f|\le1$ due to  (\ref{h_definition}),
	where the second equality follows from (\ref{h1}).
	Hence,
	\begin{equation}\label{vlessthantheintegral}
		v_{\Psi_T}(x,T-v,v)\leq\int_0^v\S_{v-l}\left[\Psi_T(\cdot,T-l)\right](x)\,dl= \int_0^v\S_l\left[\Psi_T(\cdot,T-v+l)\right](x)\,dl.
	\end{equation}
	Using that $g(s)=\frac{s^2}{2}$ and the fact that
	$\Psi_T\geq0$, we get
	\begin{eqnarray*}
		-I_4(T)&\leq&\frac{1}{2}\int_{\R^d}\int_0^T\int_0^{T-r} \S_r\left[\Psi_T(\cdot,r)\S_{T-r-v}\Big(\int_0^v\S_l\Psi_T(\cdot,T-v+l)\,dl\Big)^2\right](x)\,\mathcal{U}(T-v)\,dv\,dr\,dx.
	\end{eqnarray*}
	{By self-similarity of the semigroup $(\S_t)_{t\geq 0}$,  and
		changing the order of the integration with respect to $v$ and $r$ we obtain
		\begin{equation*}
			\begin{split}
				-I_4(T)&\leq \frac{C}{H_T^{3}}\int_{\mathbb{R}^{d}}\int_{0}^{T}\int_{0}^{T-r}\S_{T-t-v}\phi_1(x)\Bigg(\int_{0}^{T}\S_{l}\phi_1(x)dl\Bigg)^{2}\mathcal{U}(t-v)\,dv\,dr\,dx\\
				&\leq\frac{C}{H_T^{3}}\int_{\mathbb{R}^{d}}\Bigg(\int_{0}^{T}\S_{l}\phi_1(x)\,dl\Bigg)^{2}\,dx\,U(T).
			\end{split}
		\end{equation*}
		We also have
		$$\sup_{x}\Bigg(\int_{0}^{T}\S_{l}\phi_1(x)\,dl\Bigg)\leq C_1(\alpha,\phi_1,d)F_T,$$
		where
		\begin{equation*}
			\label{eq:aqui-le-mostramos-como-hacerle-la-llave-grande}
			F_T = \left\{
			\begin{array}{ll}
				1,      & \mathrm{if\ } d > \alpha, \\
				\log (T), & \mathrm{if\ } d = \alpha,  \\
				T^{1-\frac{d}{\alpha}},     & \mathrm{if\ } d < \alpha.
			\end{array}
			\right.
		\end{equation*}
		Hence, for $d>\alpha$ we can estimate
		$$-I_{4}(T)\leq \frac{C_1}{H_T^{3}}T^{\gamma}\int_{\mathbb{R}^{d}}\Bigg(\int_{0}^{T}\S_l\phi_1(x)\,dl\Bigg)^{2}\,dx=\frac{C_2}{H_T^{3}}T^{\gamma}T^{2-\frac{d}{\alpha}}=\frac{C_2}{H_T}\to 0,$$
		where $C_i\equiv C_i(\alpha,\phi_1,d,\gamma)>0$, $i=1,2$. In the same way,
		if $d<\alpha$ then
		$$-I_{4}(T)\leq \frac{C_1}{H_T^{3}}T^{\gamma}T^{2-2\frac{d}{\alpha}}\int_{\mathbb{R}^{d}}\int_{0}^{T}\S_l\phi_1(x)\,dl\,dx=C_2\frac{T^{3+\gamma-2\frac{d}{\alpha}}}{T^{3+\frac{3}{2}\gamma-\frac{3}{2}\frac{d}{\alpha}}}=C_2T^{-\frac{\gamma}{2}-\frac{d}{2\alpha}}\to 0.$$
		A similar argument woks for $d=\alpha$. This finishes the proof of (\ref{limitstocero}) for $i=4$.\medskip
		\ter
		
	}

	\noindent{\bf Step 4. Proof of (\ref{limitstocero}) for $i=2$.}   This part can be proved in a similar way as in  \cite{BGTbeta} p. 512-515. Notice that inequality (\ref{vlessthantheintegral}) implies
	$\left(\int_0^{s}\S_u\Psi_T(\cdot,T+u-s)(x)du\right)^{2}-(v_{\Psi_T}(x,T-s,s))^2\geq0.$
	Since  $g(s)=\frac{s^2}{2}$, it follows that
	$$
	0\leq -2I_2(T)=\int_0^T\int_{\R^d} \left[\left(\int_0^{s}\S_u\Psi_T(\cdot,T+u-s)(x)du\right)^{2}-(v_{\Psi_T}(x,T-s,s))^2\right]\mathcal{U}(T-s)\,dx\,ds.
	$$
	We will prove that, as $T\to\infty$,
	\begin{equation}
		\int_{0}^{T}\int_{\mathbb{R}^{d}}\LL[\LL(\int_{0}^{s}\S_u\Psi_{T}(\cdot,T+u-s)du\RR)^{2}-v_{\Psi_T}(x,T-s,s)^{2}\RR]\mathcal{U}(T-s)\,dx\,ds \to 0.
	\end{equation}
	Indeed, from (\ref{vaux}) we obtain
	\begin{eqnarray*}\lefteqn{
			-v_{\Psi_T}(x,T-s,s)}\\ &&\leq -\int_{0}^{s}\S_u\Psi_T(\cdot,T-s+u)f(\cdot,T-s+u,s-u)(x)\,du+\frac{1}{2}\int_{0}^{s}\S_uv_{\Psi_T}^{2}(x,T-s+u,s-u)\,dU(u).
	\end{eqnarray*}
	Therefore,
	\begin{eqnarray}\nonumber
		0&\leq& \int_{0}^{s}\S_u\Psi_T(\cdot,T+u-s)(x)du-v_{\Psi_T}(x,T-s,s)\\ \nonumber &\leq& \int_{0}^{s}\S_u\Psi_T(\cdot,T-s+u)\LL(1-f(\cdot,T-s+u,s-u)(x)\RR)\,du \\ \label{a1}
		&&   +\frac{1}{2}\int_{0}^{s}\S_uv_{\Psi_T}^{2}(x,T-s+u,s-u)\,dU(u).
	\end{eqnarray}
	Also, from (\ref{h1}) and (\ref{LaplaceV2}),
	\begin{equation}\label{a2}
		1-f(\cdot,T-s+u,s-u)\leq \int_{0}^{s-u}\S_w\Psi_T(\cdot,T-s+u+w)\,dw
	\end{equation}
	and
	\begin{equation}\label{a3}
		v_{\Psi_T}^{2}(\cdot,T-s+u,s-u)\leq \LL(\int_{0}^{s-u}\S_w\Psi_T(\cdot,T-s+u+w)dw\RR)^{2}.
	\end{equation}
	Thereby,
	\begin{equation}\label{a4}
		\frac{1}{2}\int_{0}^{s}\S_uv_{\Psi_T}^{2}(x,T-s+u,s-u)dU(u)\leq \frac{1}{2}\int_{0}^{s}\S_u\LL(\int_{0}^{s-u}\S_w\Psi_{T}(x,T-s+u+w)\,dw\RR)^{2}dU(u).
	\end{equation}
	From (\ref{a1}), (\ref{a2}), (\ref{a3}) and (\ref{a4}),
	\begin{eqnarray}\nonumber
		0&\leq&  \int_{0}^{s}\S_u\Psi_T(\cdot,T+u-s)(x)\,du-v_{\Psi_T}(x,T-s,s)
		\\ \nonumber
		&\leq & \int_{0}^{s}\S_u\LL(\Psi_T(\cdot,T+u-s)
		\int_{0}^{s-u}\S_w\Psi_T(\cdot,T-s+u+w)\RR)(x)\,dw\,du\\ \label{b1}
		&&+\frac{1}{2}\int_{0}^{s}\S_u\LL(\int_{0}^{s-u}
		\S_w\Psi_{T}(x,T-s+u+w)\,dw\RR)^{2}dU(u).
	\end{eqnarray}
	In addition, from (\ref{vlessthantheintegral})
	\begin{equation*}
		\begin{split}
			0&\leq \LL(\int_{0}^{s}\S_u\Psi_T(,T+u-s)(x)du\RR)^{2}-v_{\Psi_T}^{2}(x,T-s,s)\\
			&=\Bigg(\int_{0}^{s}\S_u\Psi_T(,T+u-s)(x)du-v_{\Psi_T}(x,T-s,s)\Bigg)\Bigg(\int_{0}^{s}\S_u\Psi_T(,T+u-s)(x)du+v_{\Psi_T}(x,T-s,s)\Bigg)\\
			& \leq 2\int_{0}^{s}\S_u\Psi_T(,T+u-s)(x)du\Bigg(\int_{0}^{s}\S_u\Psi_T(,T+u-s)(x)du-v_{\Psi_T}(x,T-s,s)\Bigg)\\
			&\text{where due to }(\ref{b1})\\
			&\leq 2\int_{0}^{s}\S_u\Psi_T(,T+u-s)(x)\,du\cdot\int_{0}^{s}\S_u\LL(\Psi_T(\cdot,T+u-s)\int_{0}^{s-u}\S_w\Psi_T(\cdot,T-s+u+w)\RR)(x)\,dw\,du\\
			&+\int_{0}^{s}\S_u\Psi_T(,T+u-s)(x)\,du\cdot\int_{0}^{s}\S_u\LL(\int_{0}^{s-u}\S_w\Psi_{T}(x,T-s+u+w)\,dw\RR)^{2}dU(u).
		\end{split}
	\end{equation*}
	We define
	\begin{eqnarray*}
		{R}_1(T)&=&\int_{\mathbb{R}^{d}}\int_{0}^{T}\LL(\int_{0}^{s}\S_u\LL(\Psi_T(\cdot,T+u-s)\int_{0}^{s-u}\S_w\Psi_T(\cdot,T-s+w+u)\,dw\RR)(x)du\RR)^{2}\mathcal{U}(T-s)\,ds\,dx,
		\\
		R_2(T)&=&\int_{\mathbb{R}^{d}}\int_{0}^{T}\LL(\int_{0}^{s}\S_u\LL(\int_{0}^{s-u}\S_w\Psi_T(\cdot,T-s+u+w)\,dw\RR)^{2}dU(u)\RR)^{2}\mathcal{U}(T-s)\,ds\,dx.
	\end{eqnarray*}
	Then, by the Cauchy-Schwarz inequality applied to the measure
	$\int_{\mathbb{R}^{d}}\int_{0}^{T}\mathcal{U}(T-s)\,ds\,dx$ it follows that
	\begin{equation*}
		\begin{split}
			\int_{\mathbb{R}^{d}}&\int_{0}^{T}\LL(\int_{0}^{s}\S_u\Psi_T(T+u-s)(x))^{2}-v_{\Psi_T}^{2}(x,T-s,s)\RR)\mathcal{U}(T-s)\,ds\,dx\\
			&\leq C_1\sqrt{I(T)}(\sqrt{R_1(T)}+\sqrt{R_2(T)}).
		\end{split}
	\end{equation*}
	We need to show that $R_1(T)\to0$ and $R_2(T)\to 0$ as $T\to \infty$. Indeed, for $R_1(T)$,
	\begin{equation}\nonumber
		\begin{split}
			R_1(T)&\leq C\frac{T}{H_T^{4}}\int_{\mathbb{R}^{d}}\int_{0}^{1}\LL(\int_{0}^{Ts}\S_u\LL(\phi_1(\cdot)\int_{0}^{Ts-u}\S_w\phi_1(\cdot)\,dw\RR)(x)\,du\RR)^{2}\mathcal{U}(T(1-s))\,ds\,dx\\
			&\leq C\frac{T^{4}U(T)}{H_T^{4}}\int_{\mathbb{R}^{d}}\LL(\int_{0}^{1}\S_{Tu}\LL(\phi_1(\cdot)\int_{0}^{1}\S_{Tw}\phi_1(\cdot)\,dw\RR)(x)\,du\RR)^{2}dx.
		\end{split}\phantom{XXXXXXXXXX}
	\end{equation}
	Following similar arguments as in \cite[(3.30)-(3.33)]{BGTbeta}, from here we can deduce that $\lim_{T\to \infty}R_1(T)=0.$
	
	We now work the term $R_2(T)$. We define 
	$$r(x)=\int_{0}^{1}p_u(x)\,du,\quad f_{1,T}(x)=\int_{0}^{1}p_{u,\alpha}(x)\frac{\mathcal{U}(Tu)}{T^{\gamma-1}}\,du,\quad  g_{1,T}(x)=T^{\frac{d}{\alpha}}\phi_1(T^{\frac{1}{\alpha}}x).$$
	Here $\|f_{1,T}\|_{1}\leq \frac{U(T)}{T^{\gamma}}$, which is bounded uniformly in $T$ for $T$ sufficiently large. Moreover, it can be shown, as in \cite{BGTbeta}, that  $\|g_{1,T}\|_{1}=\|\phi_1\|_{1}<\infty$ and $\|r\|_2<\infty$. Making the change of variables $s'=\frac{s}{T}$ gives
	$$R_2(T)\leq C\frac{T}{H_{T}^{4}}\int_{\mathbb{R}^{d}}\int_{0}^{1}\LL(\int_{0}^{Ts}\S_u\LL(\int_{0}^{Ts-u}\S_w\phi_1(\cdot)\,dw\RR)^{2}(x)\mathcal{U}(u)du\RR)^{2}\mathcal{U}(T(1-s))\,ds\,dx.$$
	Making the change of variables $u'=\frac{u}{T}$ yields
	$$R_2(T)\le C\frac{T^{2+1}}{H_T^{4}}\int_{\mathbb{R}^{d}}\int_{0}^{1}\LL(\int_{0}^{s}\S_{Tu}\LL(\int_{0}^{T(s-u)}\S_w\phi_1(\cdot)\,dw\RR)^{2}(x)\mathcal{U}(Tu)\,du\RR)^{2}\mathcal{U}(T(1-s))\,ds\,dx.$$
	Making the change of variables $w'=\frac{w}{T}$ renders
	\begin{eqnarray*}R_2(T)
		&\le&C\frac{T^{2+1+(2)(2)}}{H_T^{4}}\int_{\mathbb{R}^{d}}\int_{0}^{1}\LL(\int_{0}^{s}\S_{Tu}\LL(\int_{0}^{s-u}\S_{Tw}\phi_1(\cdot)\,dw\RR)^{2}(x)\mathcal{U}(Tu)\,du\RR)^{2}\mathcal{U}(T(1-s))\,ds\,dx\\
		&\leq& C\frac{T^{2+1+(2)(2)}}{H_T^{4}}\frac{U(T)}{T}\int_{\mathbb{R}^{d}}\LL[\int_{0}^{1}\S_{Tu}\LL(\int_{0}^{1}\S_{Tw}\phi_1(\cdot)\,dw\RR)^{2}(x)\frac{\mathcal{U}(Tu)}{T^{\gamma-1}}du\RR]^{2}\,dx.
	\end{eqnarray*}
	By self-similary property of $(\S_u)_{u\geq 0}$ and making the changes of variables  $x'=T^{-\frac{1}{\alpha}}x,$\ $y'=T^{-\frac{1}{\alpha}}y,$ and $ z'=T^{-\frac{1}{\alpha}}z$,
	\begin{eqnarray*}R_2(T)
		&\le&C\frac{U(T)}{T^{\frac{d}{\alpha}}}\int_{\mathbb{R}^{d}}\LL[\int_{\mathbb{R}^{d}}\int_{0}^{1}p_{u,\alpha}(x-y)\frac{\mathcal{U}(Tu)}{T^{\gamma-1}}\,du\LL(\int_{\mathbb{R}^{d}}\int_{0}^{1}p_{w,\alpha}(y-z)\,dw T^{\frac{d}{\alpha}}\phi_1(T^{\frac{1}{\alpha}}z)\,dz\RR)^{2}dy\RR]^{2}\,dx\\
		&=&C\frac{U(T)}{T^{\frac{d}{\alpha}}}\|f_{1,T}*(r*g_{1,T})^{2}\|_{2}^{2},
	\end{eqnarray*}
	and $\|f_{1,T}\|_{1}$ is bounded uniformly in $T$ for $T$ sufficiently large. Hence
	$$
	R_2(T)\
	\leq\ C\frac{U(T)}{T^{\frac{d}{\alpha}}}\|f_{1,T}\|_{1}^{2}\|r*g_{1,T}\|_{2}^{2}
	\leq C\frac{U(T)}{T^{\frac{d}{\alpha}}}\|f_{1,T}\|_{1}^{2}\|r\|_{2}^{4}\|g_{1,T}\|_{1}^{4},$$
	where we have used Young's inequality.
	Again, as $\|g_{1,T}\|_1=\|\phi_1\|_1$ and $\alpha\gamma<d<\alpha(1+\gamma)$,  we will have $\frac{U(T)}{T^{\frac{d}{\alpha}}}\overset{T\to \infty}{\to}0$ and $\|r\|_{2}<\infty$. Therefore, $\lim_{T\to \infty}R_2(T)=0.$
	
	We have proved that both $R_1(T)$ and $R_2(T)$ tend to $0$ as $T \to \infty$. This proves Step 4 and shows that (\ref{limitstocero}) holds.\hfill$\Box$
	\begin{lemma}  Under the assumptions in Proposition \ref{fdd_convergence_prop}, the limit (\ref{fdd_convergence_eq}) can be written as
		\begin{eqnarray}\label{ffd_convergence_eq1}
			\lefteqn{\lim_{T\rightarrow \infty}\E\left[e^{-\langle\Psi,\tilde{\J}_T\rangle}\right]} \nonumber
			\\
			&=&\begin{cases}\exp\left(\frac{\gamma\langle\phi_1,\lambda\rangle^{2}}{\Gamma(\gamma+1)(2\pi)^d(2-\frac{d}{\alpha})}\int_{\mathbb{R}^{d}}e^{-|y|^{\alpha}}dy\int_{0}^{1}\int_{0}^{1}Q(w,z)\phi_2(w)\phi_2(z)\,dw\,dz \right), &\mbox{ if }d\neq \alpha, \\
				$\,$\\
				\exp\left(\frac{\gamma\langle\phi_1,\lambda\rangle^{2}}{2\Gamma(\gamma+1)(2\pi)^d}\int_{\mathbb{R}^{d}}e^{-|y|^{\alpha}}dy\int_{0}^{1}\int_{0}^{1}K(w,z)\phi_2(w)\phi_2(z)\,dw\,dz \right),&\mbox{ if }d=\alpha,
			\end{cases}
		\end{eqnarray}
		where
		\begin{equation}\label{Q}
			Q(w,z)=\LL(\frac{d}{\alpha}-1\RR)^{-1}\int_{0}^{z\wedge w}s^{\gamma-1}\left[(w\wedge z-s)^{2-\frac{d}{\alpha}}+(w\vee z-s)^{2-\frac{d}{\alpha}}-(w+z-2s)^{2-\frac{d}{\alpha}}\right]ds
		\end{equation}
		and
		\begin{equation*}
			K(w,z):=\frac{1}{2}\int_{0}^{z\wedge w}s^{\gamma-1}\Bigg[(w+z-2s)\ln(w+z-2s)-(w-s)\ln(w-s)-(z-s)\ln(z-s)\Bigg]ds.
		\end{equation*}
	\end{lemma}
	\proof We first deal with the case $d\neq\alpha$. Recall that $\chi(u)=\int_u^1\phi_2(w)\,dw$. Substituting this into the triple integral in the  right hand side of (\ref{fdd_convergence_eq}), by symmetry of the function $(w,z)\mapsto \phi_2(w)\phi_2(z)$ and changing the order of integration we conclude that
	\begin{equation}\label{QC}
		\int_{0}^{1}\int_{0}^{u}\int_{0}^{v}s^{\gamma-1}(u+v-2s)^{-\frac{d}{\alpha}}\chi(u)\chi(v) \, ds \, dv \, du=\frac{1}{2(2-\frac{d}{\alpha})}\int_{0}^{1}\int_{0}^{1}Q(w,z)\phi_2(w)\phi_2(z) \, dw \, dz ,
	\end{equation}
	where
	\begin{equation*}
		Q(w,z)=\LL(\frac{d}{\alpha}-1\RR)^{-1}\int_{0}^{z\wedge w}s^{\gamma-1}\Bigg[(w-s)^{2-\frac{d}{\alpha}}+(z-s)^{2-\frac{d}{\alpha}}-(w+z-2s)^{2-\frac{d}{\alpha}}\Bigg] \, ds.
	\end{equation*}
	For the case  $d=\alpha$, similarly as above we can show that
	\begin{equation}\label{KC}
		\int_{0}^{1}\int_{0}^{u}\int_{0}^{v}s^{\gamma-1}(u+v-2s)^{-\frac{d}{\alpha}}\chi(u)\chi(v) \, ds \, dv \, du=\frac{1}{2}\int_{0}^{1}\int_{0}^{1}K(w,z)\phi_2(w)\phi_2(z) \, dw \, dz ,
	\end{equation}
	with
	\begin{equation*}
		K(w,z):=\frac{1}{2}\int_{0}^{z\wedge w}s^{\gamma-1}\Bigg[(w+z-2s)\ln(w+z-2s)-(w-s)\ln(w-s)-(z-s)\ln(z-s)\Bigg] \, ds.
	\end{equation*}
	The proof is finished because the expressions (\ref{QC}) and (\ref{KC}) are equivalent to (\ref{fdd_convergence_eq}) for $d\neq \alpha$ and $d=\alpha$ respectively.
	\ter
	
	\noindent{\bf Proof of Theorem \ref{main1}}. Proposition \ref{ThightnessProp} gives the tightness and Proposition \ref{fdd_convergence_prop} identifies uniquely any limit point of $\{\J_T$, $T\ge0\}$ for non-negative test functions. For general test functions the proof can be done as  in \cite[page 9]{BGT1}. For the sake of brevity we omit the details. \hfill$\Box$
	
	\medskip
	
	\noindent{\bf Proof of Theorem \ref{main}}. The proof of this result can be done following the same lines as the proof of Theorem \ref{main1} but using the fact that, in the case of lifetimes with finite mean  $\mu$, the renewal measure is such that
	\[\frac{U(Tr)}{T}\rightarrow \frac{r}{\mu}\quad\mbox{as $T\to\infty$ for all $r>0$}. \]
	Formally, this can be thought as putting $\gamma=1$ in all the preceding computations.\hfill$\Box$

	\subsection{\bf Proof of Theorem \ref{QgeneralTh}.}

	\noindent{\bf Proof of  (i).}
	Note that,  for $b=2$
	\begin{eqnarray*}
		Q_{a,2}(w,z)&=& 2\int_{0}^{z\wedge w}s^{a}(z-s)(w-s) \, ds, \end{eqnarray*}
	which is a  positive definite function and it is finite if and only if $a>-1$. For the case $b=0$ we have that
	\[Q_{a,0}(w,z)=\frac{1}{a+1}(s\wedge t)^{a+1},\]
	which, for $a>-1$,  corresponds to the covariance function of a time-changed Brownian motion.
	
	Let us consider the case  of $a>-1$ and $0<b<2$, with $b\neq1$.   Observe that
	(\ref{Qgeneral}) can be written as
	\begin{equation*}
		Q_{a,b}(w,z)=\int_0^{w\wedge z } s^a\kappa_b(w-s,z-s)\,ds,
	\end{equation*}
	where \begin{eqnarray*}
		\kappa_b(w,z)=\frac{1}{1-b}\left(w^b+z^b-(w+z)^b\right),
	\end{eqnarray*}
	which is a covariance function for $b\in(0,1)\cup (1,2)$,  see  (2.4) in \cite{BGT-quasi} or (1.1) (with $K=1$) in \cite{Lei-Nualart}. Therefore, $Q_{a,b}$ is a covariance function for $a>-1$ and $b\in(0,1)\cup (1,2)$.
	\ter
	
	\noindent{\bf Proof of  (ii).} Consider $a>-1$ and $-1<b<0$ such that $a+b+1\geq 0$. Note that (\ref{Qgeneral}) can be written as
	\begin{equation}\label{Q-decomp}
		Q_{a,b}(w,z)=\frac{1}{1-b}\left(Q_1(w,z)+Q_2(w,z)\right),
	\end{equation}
	where\begin{equation}\label{Q1}
		Q_1(w,z):=\int_0^{w\wedge z }u^a(w\wedge z-u)^b du=(w\wedge z)^{a+b+1} \int_0^1u^a(1 -u)^b du,
	\end{equation}
	and \begin{equation}
		Q_2(w,z):=\int_0^{w\wedge z} u^a\left[ (w\vee z -u)^b-(w\wedge z +w\vee z-2u)^b\right] du.
	\end{equation}
	Note that,
	\begin{eqnarray*}
		\lefteqn{Q_2(w,z)}\\&=& -b\int_0^{w\wedge z} \int _0^{w\wedge z -u}u^a(w\vee z -u+r)^{b-1} dr du
		\\&=&-b\int_0^{w\wedge z} \int_0^{(w- u)\wedge (z-u)} u^a (w-u+r)^{b-1}\wedge(z-u+r)^{b-1} dr du
		\\&=& -b \int_0^\infty\int_0^\infty\int_0^\infty u ^a \mathbf{1}_{[0,w]}(u) \mathbf{1}_{[0,w-u]}(r)
		\mathbf{1}_{[0,(w-u+r)^{b-1}]}(v) \mathbf{1}_{[0,z]}(u) \mathbf{1}_{[0,z-u]}(r) \mathbf{1}_{[0,(z-u+r)^{b-1}]}(v) \, dv\, dr \, du,
	\end{eqnarray*}
	thus $Q_2$ also is positive definite. Clearly (\ref{Q1}) is non-negative definite if $a+b+1\geq 0$. Then, from (\ref{Q-decomp}) it follows that $Q_{a,b}$ is a covariance function for $a>-1$ and  $-1<b<0$ such that $a+b+1\geq 0$.\ter 
	
	\subsection{ Proof of Lemma \ref{QnotConvariance}}
	In order to prove Lemma \ref{QnotConvariance} we notice that, for $a>-1$ and $b>-1$,
	\begin{equation}\label{App1}
		Q_{a,b}(t,t)=\frac{2-2^b}{1-b}t^{a+b+1}\int_0^1u^a(1-u)^bdu\equiv\frac{2-2^b}{1-b}\mathcal{B}(a+1,b+1)t^{a+b+1}.
	\end{equation}
	The restrictions $a>-1$ and $b>-1$ are necessary for the integral above to be finite, and any $a$ or $b$ out of this range is ruled out. Moreover, for $1\leq t$,
	\begin{equation}\label{App2}
		Q_{a,b}(1,t)=\frac{t^{a+b+1}}{1-b}\int_0^{1/t}u^a(1-u)^bdu+\frac{1}{1-b}\mathcal{B}(a+1,b+1)
		-\frac{t^{a+1}}{1-b}\int_0^{1/t}u^a(t+1-2tu)^{b}du,
	\end{equation}
	whereas for $1\geq t$,
	\begin{eqnarray}
		Q_{a,b}(1,t)&=&\frac{t^{a+b+1}}{1-b} \int_0^1u^a(1-u)^b \,du+\frac{1}{1-b}\int_0^t u^a\left[(1-u)^b-(1+t-2u)^b\right]\, du.
	\end{eqnarray}
	Note that, for the case $b>0$
	\begin{equation}\label{Q-b-positive}
		Q_{a,b}(w,z)=b\int_0^{w\wedge z}  \int_u^{w\vee z}  \int_u^{w\wedge z}  u^a (r+v-2u)^{b-2}\, dv\, dr\, du.
	\end{equation}

	\noindent{\bf Proof of  (i). }
	For $a>-1$ and $-1<b<0$, with $a+b+1< 0$,  the function $Q_{a,b}(\cdot,\cdot)$ is not a covariance. In fact,  from (\ref{App1}) we have that
	\[\sqrt{Q_{a,b}(1,1)Q_{a,b}(t,t)}=\frac{2-2^b}{1-b}\mathcal{B}(a+1,b+1)t^{\frac{a+b+1}{2}}.\] On the other hand, for $0<t<1$ we have that
	\begin{eqnarray}
		Q_{a,b}(1,t)&=&\frac{t^{a+b+1}}{1-b} \int_0^1u^a(1-u)^b \,du+\frac{1}{1-b}\int_0^t u^a\left[(1-u)^b-(1+t-2u)^b\right]\, du
		\\&\geq &\frac{t^{a+b+1}}{1-b} \int_0^1u^a(1-u)^b \,du,
	\end{eqnarray}
	where to get the inequality we have used that the function $u\mapsto (1-u)^b-(1+t-2u)^b\geq 0$ since $-1<b<0$.
	Therefore, whenever $a+b+1<0$, as $t\downarrow0$, $Q_{a,b}(w,z)$ does not satisfy the inequality covariance
	\begin{equation}\label{cov-inequality}
		Q_{a,b}(1,t)\leq \sqrt{Q_{a,b}(1,1)Q_{a,b}(t,t)}.
	\end{equation}
	\ter
	
	\noindent{\bf Proof of  (ii). }
	Take $a>-1$ and $b>2$. Assume first $b>a+3$. From (\ref{App1}) we have that
	\begin{equation}
		\sqrt{Q_{a,b}(1,1)Q_{a,b}(t,t)}=\frac{2^{b}-2}{(b-1)}\mathcal{B}(\alpha+1,b+1)t^{\frac{a+b+1}{2}}.
	\end{equation}
	On the other hand,  from (\ref{Q-b-positive}) it follows that for  $t>1$,
	\begin{eqnarray}
		Q_{a,b}(1,t) &=&   b\int_0^{1}  \int_u^{t}  \int_u^{1}  u^a (r+v-2u)^{b-2}\, dv\, dr\, du\nonumber
		\\ &=& bt^{b-1}\int_0^{1}  \int_{\frac{u}{t}}^{1}  \int_u^{1}  u^a \left(r+\frac{v}{t}-2\frac{u}{t}\right)^{b-2}\, dv\, dr\, du,
	\end{eqnarray}
	which implies that,
	\begin{eqnarray}\label{Qb-positive-big-t}
		\lim_{t\rightarrow\infty }\frac{Q_{a,b}(1,t)}{t^{b-1}}= b\int_0^{1} du  \int_{0}^{1}  dr \int_u^{1} dv  u^a r^{b-2}=\frac{b}{b-1}\int_0^{1} u^a(1-u) du.
	\end{eqnarray}
	Hence, as $t\uparrow\infty$, from (\ref{Qb-positive-big-t}) the left-hand side of (\ref{cov-inequality}) is of order $t^{b-1}$, whereas the right-hand side of (\ref{cov-inequality}) is of order $t^{\frac{a+b+1}{2}}$. Thus, $Q_{a,b}(\cdot,\cdot)$ can not be a covariance function  for $a>-1$ and $b>a+3$, since $b>a+3$ implies that $b-1>(a+b+1)/2$.
	\ter

\subsection{ Proof of Theorem \ref{limit-properties} and Theorem \ref{limit-Bm}}

\subsubsection{Proof of Theorem \ref{limit-properties}}
Since $\zeta$ is a Gaussian process, the proofs are based on properties of its covariance function $Q_{a,b}$ given by (\ref{Qgeneral}).

\noindent{\bf Proof of (i).}
Let  $c$ be a positive constant and  $t\geq 0$. Then,

\begin{eqnarray*}
	Q_{a,b}(ct,ct)&=&\frac{1}{1-b}\int_{0}^{ct}s^{a}\left((ct-s)^{b}+(ct-s)^{b}-(2ct-2s)^{b}\right)\,ds
	=\frac{2-2^{b}}{1-b}\int_{0}^{ct}s^{a}(ct-s)^{b}\,ds\\
	& =& c^{b+a+1}Q_{a,b}(t,t).
\end{eqnarray*}\hfill$\Box$

\noindent{\bf Proof of (ii).} From (\ref{Qgeneral}) it  follows easily that 
\begin{eqnarray}\nonumber
	\E\left[(\zeta(t)-\zeta(s))^{2}\right]
	&=&\frac{1}{1-b}\Bigg[2\int_{s}^{t}u^{a}(t-u)^{b}\,du+2\int_{0}^{s}u^{a}(t+s-2u)^{b}\,du\\ \label{Cov1}
	&&\phantom{XXX}-2^{b}\left(\int_{0}^{t}u^{a}(t-u)^{b}du+\int_{0}^{s}u^{a}(s-u)^{b}du\right)\Bigg]\\ \nonumber
	&=&\frac{2^{b}-2}{b-1}\int_{s}^{t}u^{a}(t-u)^{b}\,du\\ \label{Cov2}
	&&+\frac{1}{b-1}\int_{0}^{s}u^{a}\left[2^{b}(t-u)^{b}+2^{b}(s-u)^{b}-2(t+s-2u)^{b}\right]\,du.
\end{eqnarray}

{\noindent\bf (a).} 
Suppose that $b \in (1,2]$, $-1<a<0$ and $0\leq s<t\leq M$ with $0<t-s\leq 1$, where $M>0$ is a constant. For this case we have in mind (\ref{Cov2}).
For $0\leq u\leq s$ we define
\[r_{t,s}(u)=2^{b}(t-u)^{b}+2^{b}(s-u)^{b}-2(t+s-2u)^{b},\]
hence
$$\frac{dr_{t,s}(u)}{du}=-2^{b}b(t-u)^{b-1}-2^{b}b(s-u)^{b-1}+4b(t+s-2u)^{b-1}.$$
Since $0<b-1\leq1$, the function $u\mapsto u^{b-1}$ is  concave, which implies that  $r_{t,s}$ is non decreasing  and  $r_{t,s}(u)\leq (2^{b}-2)(t-s)^{b}$ for all $ u \in [0,s]$. Therefore,
\[\frac{1}{b-1}\int_{0}^{s}u^{a}\Bigg[2^{b}(t-u)^{b}+2^{b}(s-u)^{b}-2(t+s-2u)^{b}\Bigg]du\leq c_{a,b}(M)(t-s)^{b},\]
where $c_{a,b}(M)=\frac{(2^{b-1}-2)M^{a+1}}{(b-1)(a+1)}.$
On the other hand
\begin{equation}\label{pp22}
	\int_s^t u^a(t-u)^{b}du \leq  (t-s)^{b}\int_{s}^{t}u^{a}du
	\leq  (t-s)^{b}\LL(\frac{t^{a+1}-s^{a+1}}{a+1}\RR)
	\leq  c_a(t-s)^{a+b+1},
\end{equation}
where the last inequality is obtained using that   $t\to t^{a+1}$ is $(a+1)$-Hölder continuous and $c_a$ is a positive constant.   Since $b<a+b+1$ and  $0<t-s\leq 1$, we have shown that
\begin{equation*}
	\E\left[(\zeta(t)-\zeta(s))^{2}\right]\leq \kappa |t-s|^{b},\,\,\,\mbox{for some constant $\kappa>0$.}
\end{equation*}
Suppose now that $0<b<1$. In this case,
\begin{equation}\label{pp12}
	\begin{split}
		\E\left[(\zeta(t)-\zeta(s))^{2}\right]
		&=\frac{2-2^{b}}{1-b}\int_{s}^{t}u^{a}(t-u)^{b}du\\
		&+\frac{1}{1-b}\int_{0}^{s}u^{a}\left[2(t+s-2u)^{b}-2^{b}(t-u)^{b}-2^{b}(s-u)^{b}\right]\,du,
	\end{split}
\end{equation}
with \[\frac{2-2^{b}}{1-b}\int_{s}^{t}u^{a}(t-u)^{b}\,du\leq \frac{2-2^{b}}{1-b}(t-s)^b\int_{s}^{t}u^{a}\,du=
\frac{2-2^{b}}{1-b}(t-s)^b\LL(\frac{t^{a+1}-s^{a+1}}{a+1}\RR)\le
c'_{a,b} (t-s)^{a+b+1},\]
where $c'_{a,b}>0$ is a constant. 
Using that $2(t+s-2u)^{b}-2^{b}(t-u)^{b}-2^{b}(s-u)^{b}\leq (2-2^b)(t-s)^b$ for $0\leq u\leq s$, we get that the second integral in (\ref{pp12}) is bounded from above by 
$$
\frac{2-2^b}{1-b} (t-s)^b\int_0^s u^a du=\frac{2-2^b}{1-b} \frac{s^{a+1}}{a+1}(t-s)^b.
$$
It follows that the process $\zeta$ is locally $\delta$-Hölder continuous for  $0<\delta<b/2$.\bigskip

\noindent{\bf (b).} Let $-1<b\leq0$ and $a+b>0$. In this case we work with (\ref{Cov1}). Notice that
\begin{eqnarray}\label{pp22}
	\int_{s}^{t}u^{a}(t-u)^{b}\,du\leq t^{a}\int_{s}^{t}(t-u)^{b}\,du=\frac{t^{a}}{b+1}(t-s)^{b+1}\leq C_{a,b}(t-s)^{b+1}
\end{eqnarray}
and
\begin{eqnarray}\nonumber
	\int_{0}^{s}u^{a}(t+s-2u)^{b}du&=&\left(\frac{1}{2}\right)^{a+1}\int_{0}^{2s}u^{a}(t+s-u)^{b}\,du=\left(\frac{1}{2}\right)^{a+1}(t+s)^{a+b+1}\int_{0}^{\frac{2s}{t+s}}u^{a}(1-u)^{b}du\nonumber\\ \nonumber
	&\leq& \left(\frac{1}{2}\right)^{a+1}(t+s)^{a+b+1}{\cal B}(a+1,b+1)\\ \label{pp32}
	&\le&2^{b-1}\LL(t^{a+b+1} + s^{a+b+1}\RR){\cal B}(a+1,b+1)
\end{eqnarray}
because the mapping $t\to t^{a+b+1}$ is convex due to $a+b+1>1$.
Also
\begin{equation}\label{pp42}
	\int_{0}^{r}u^{a}(r-u)^{b}\,du=r^{a+b+1}\int_{0}^{1}u^{a}(1-u)^{b}\,du=r^{a+b+1}{\cal B}(a+1,b+1),\quad r\in\{s,t\}.
\end{equation}
Plugging (\ref{pp22})-(\ref{pp42}) into (\ref{Cov1}) yields
\begin{eqnarray*}
		\E\left[(\zeta(t)-\zeta(s))^{2}\right]  &\le&
		\frac{2C_{a,b}}{1-b}
		(t-s)^{b+1} \le \frac{2C_{a,b}}{1-b}|t-s|^{b+1}.
	\end{eqnarray*}
	Thus, $\zeta$ is $\delta$-Hölder continuous for any $0<\delta< (b+1)/2$. \ter

	\noindent{\bf Proof of (iii).} Follows immediately from  (\ref{Qgeneral}).\ter
	
	\noindent{\bf Proof of (iv).} Let us first show that \begin{equation}\label{LRD-firstorder}
		\lim_{T\rightarrow \infty}T^{1-b}\mathcal{Q}(r,v,s+T,t+T)=0.
	\end{equation}  Using (iii), the equalities
	$$
	T^{1-b}\left(\frac{(t+T-u)^{b}-(s+T-u)^{b}}{b}\right)=T^{1-b}\int_{s+T}^{t+T}(w-u)^{b-1}\,dw
	=\int_{s}^{t}\left(\frac{w+T-u}{T}\right)^{b-1}\,dw
	$$
	and the bounded convergence theorem, we obtain
	\begin{equation*}
		\lim_{T\to\infty}  T^{1-b}\int_{r}^{v}u^{a}\left[(t-u)^{b}-(s-u)^{b}\right]du=\frac{b}{a+1}(t-s)\left(v^{a+1}-r^{a+1}\right).
	\end{equation*}
	Similarly we get
	\begin{equation*}
		\lim_{T\to\infty}  T^{1-b}\int_{0}^{r}u^{a}\left[(t+r-2u)^{b}-(s+r-2u)^{b}\right]du=\frac{b}{a+1}(t-s)r^{a+1}.
	\end{equation*}
	The limit (\ref{LRD-firstorder}) follows from
	\begin{equation*}
		\begin{split}
			\lim_{T\rightarrow\infty}T^{1-b}&        \left[\int_{r}^{v}u^{a}\left((t-u)^{b}-(s-u)^{b}\right)du+\int_{0}^{r}u^{a}\left((t+r-2u)^{b}-(s+r-2u)^{b}\right)du\right.\\
			&\left.\phantom{MMMMMMMMMMMMM}-\int_{0}^{v}u^{a}\left((t+v-2u)^{b}-(s+v-2u)^{b}\right)du\right]
			\\ &=\frac{b}{a+1}(t-s)\left(v^{a+1}-r^{a+1}+r^{a+1}-v^{a+1}\right)
			=0.
		\end{split}
	\end{equation*} Now, observe that
	\[\lim_{T\rightarrow \infty}T^{2-b}\mathcal{Q}(r,v,s+T,t+T)=\lim_{T\rightarrow \infty}\frac{T^{1-b}\mathcal{Q}(r,v,s+T,t+T)}{T^{-1}}. \]
	Due to (\ref{LRD-firstorder}) we can use L'Hospital's theorem to calculate the last limit. Recall that,
	\begin{eqnarray*}
		(1-b) \mathcal{Q}(r,v,s+T,t+T)&=&\int_r^v u^a\left[(T+t-u)^b-(T+s-u)^b\right] du
		\\&&-\int_0^v u^a\left[(T+t+v-2u)^b-(T+s+v-2u)^b\right] du
		\\&&+\int_0^r u^a\left[(T+t+r-2u)^b-(T+s+r-2u)^b\right] du
		\\&=&b\int_r^v u^a \int_s^t (T+h-u)^{b-1} dh du -b\int_0^vu^a\int_s^t (T+h+v-2u)^{b-1}\,dh\,du
		\\&&+b\int_0^ru^a\int_s^t (T+h+r-2u)^{b-1}\,dh\,du.
	\end{eqnarray*}
	Therefore,
	\begin{eqnarray*}\lefteqn{
			(1-b) T^{1-b} \mathcal{Q}(r,v,s+T,t+T)}\\
		&=&b\int_r^v u^a \int_s^t \left(1+\frac{h-u}{T}\right)^{b-1} \, dh  \, du -b\int_0^vu^a\int_s^t
		\left(1+\frac{h+v-2u}{T}\right)^{b-1} \, dh \, du
		\\&&+b\int_0^ru^a\int_s^t \left(1+\frac{h+r-2u}{T}\right)^{b-1} \, dh \, du.
	\end{eqnarray*}
	Applying L'Hospital's rule we have that
	\begin{eqnarray*}\lefteqn{
			\lim_{T\rightarrow \infty}(1-b)T^{2-b}\mathcal{Q}(r,v,s+T,t+T)}\\
		&=&\lim_{T\rightarrow\infty}\left[- b(b-1)\int_r^v u^a \int_s^t \left(1+\frac{h-u}{T}\right)^{b-2} (h-u) \, dh \,  du \right]
		\\&& +\lim_{T\rightarrow\infty}\left[b(b-1)\int_0^vu^a\int_s^t
		\left(1+\frac{h+v-2u}{T}\right)^{b-2}(h+v-2u) \, dh \, du\right]
		\\&&- \lim_{T\rightarrow\infty}\left[b(b-1) \int_0^ru^a\int_s^t \left(1+\frac{h+r-2u}{T}\right)^{b-2}(h+r-2u) \, dh \, du\right],
	\end{eqnarray*}
	where \begin{eqnarray*}
		\lefteqn{ - b(b-1)\int_r^v u^a \int_s^t \left(1+\frac{h-u}{T}\right)^{b-2} (h-u) \, dh  \, du}\\&\xrightarrow{T\rightarrow\infty}&-b(b-1)\int_r^vu^a\int_s^t(h-u) \, dh \, du
		\\&=&-b(b-1)\left[\frac{t^2-s^2}{2}\frac{v^{a+1}-r^{a+1}}{a+1}-(t-s)\frac{v^{a+2}-r^{a+2}}{a+2}\right].
	\end{eqnarray*}
	Similarly, one can see that
	\begin{eqnarray*}
		\lefteqn{b(b-1)\int_0^vu^a\int_s^t
			\left(1+\frac{h+v-2u}{T}\right)^{b-2}(h+v-2u) \, dh \, du}
		\\&\xrightarrow{T\rightarrow\infty}& b(b-1)\left[\frac{t^2-s^s}{2}\frac{v^{a+1}}{a+1}-a(t-s)\frac{v^{a+2}}{(a+1)(a+2)}\right],
	\end{eqnarray*}and
	\begin{eqnarray*}
		\lefteqn{-b(b-1) \int_0^ru^a\int_s^t \left(1+\frac{h+r-2u}{T}\right)^{b-2}(h+r-2u) \, dh \, du}
		\\&\xrightarrow{T\rightarrow\infty}& -b(b-1)\left[\frac{t^2-s^s}{2}\frac{r^{a+1}}{a+1}-a(t-s)\frac{r^{a+2}}{(a+1)(a+2)}\right].
	\end{eqnarray*}
	Putting all these limits together we obtain (\ref{LRD-identity}).
	$\,$\hfill$\Box$
	
	\noindent{\bf Proof of (v).} The result follows using \cite[Proposition 13.7]{Kallenberg}  (see also \cite[Section III.8]{Feller2}) and the fact that the function $Q_{a,b}$ given in (\ref{Qgeneral}) does not satisfy
	$$Q_{a,b}(s,t)=\frac{Q_{a,b}(s,r)Q_{a,b}(r,t)}{Q_{a,b}(r,r)},\quad s\leq r\leq t.
	$$ $\,$\hfill$\Box$

	\subsubsection{Proof of Theorem \ref{limit-Bm}}
	
	Since $\zeta$ is a Gaussian process, to prove the desired convergence  it suffices to show convergence of the covariance functions of the rescaled processes. 
	Suppose that $0<s\leq t$ and $b\in(0,1)\cup (1,2]$. It is not difficult to check that
	\begin{eqnarray}\nonumber
		\lefteqn{
			E\LL[(\zeta(s+T)-\zeta(T))(\zeta(t+T)-\zeta(T))\RR]}\\ \nonumber
		&=& b\Bigg(\int_{T}^{s+T}\int_{r}^{t+T}\int_{r}^{s+T}(u+v-2r)^{b-2}r^{a}\,du\,dv\,dr+\int_{0}^{T}\int_{r}^{t+T}\int_{r}^{s+T}(u+v-2r)^{b-2}r^{a}\,du\,dv\,dr\Bigg)\\ \nonumber
		&=:&b(J_1+J_2).
	\end{eqnarray}
	Observe that 
	\begin{equation}\label{covf}
		T^{-a}J_1=\int_0^{s}\int_{r}^{t}\int_{r}^{s}(u+v-2r)^{b-2}\Bigg(\frac{r+T}{T}\Bigg)^{a}\,du\,dv\,dr \to \int_{0}^{s}\int_{r}^{t}\int_{r}^{s}(u+v-2r)^{b-2}\,du\,dv\,dr.
	\end{equation}
	The limit in  (\ref{covf}) corresponds to the covariance of the subfractional Brownian motion with parameter $b+1$, for $b\in(0,1)$ (\cite{BGTWeighted}). Moreover, it is the covariance of negative subfractional Brownian motion if $b\in(1,2)$ (\cite{BGTWeighted}).  
	
	\noindent{\bf Proof of (i). } If $b\in (1,2]$ then, for $0<s\leq t$, we have
	\begin{eqnarray*}
		T^{-a-(b-1)}J_2&=&\int_0^{t}\int_0^{s}2^{b-2}\int_{0}^{1}\left(\frac{u+v}{2T}+1-r\right)^{b-2}r^{a}\,dr\,du\,dv\nonumber
		\\&\overset{T\to \infty}{\to}& 2^{b-2}{\cal B}(a+1,b-1)st.
	\end{eqnarray*}
	This, together with  (\ref{covf}) implies that 
	\[T^{-a-(b-1)}b(J_1+J_2)\overset{T\to \infty}{\to} 2^{b-2}b {\cal B}(a+1,b-1)st,\quad T\rightarrow\infty,\]
	which finishes the proof.
	\ter 
	
	\noindent{\bf Proof of (ii).} If $b\in (0,1)$, $a>-1$ and $a+b+1>0$ then
	\begin{equation}
		\begin{split}
			T^{-a}J_2&=\frac{1}{2}\int_{0}^{s}\int_{0}^{t}\int_{0}^{2T}(u+v+r)^{b-2}\Bigg(\frac{2T-r}{2T}\Bigg)^{a} \, dr \, du \, dv\\
			&=\frac{1}{2}\int_{0}^{s}\int_{0}^{t}\int_{u+v}^{u+v+2T}r^{b-2}\Bigg(\frac{u+v+2T-r}{2T}\Bigg)^{a} \, dr \, du \, dv.
		\end{split}
	\end{equation}
	It is easy to show that
	\begin{equation}\label{J_2}
		T^{-a}J_2\to \frac{1}{2}\int_{0}^{s}\int_{0}^{t}\int_{u+v}^{\infty}r^{b-2} \, dr \, du \, dv=\frac{1}{2(1-b)}\int_{0}^{s}\int_{0}^{t}(u+v)^{b-1} \, du \, dv.   
	\end{equation}
	Therefore, 
	$$T^{-a}b(J_1+J_2)\to \frac{1}{2(b+1)(1-b)}\Bigg(t^{b+1}+s^{b+1}-(t-s)^{b+1}\Bigg)\quad \mbox{ as }
	T\to \infty.$$
	The processes converge to a fractional Brownian motion with Hurst parameter $H = \frac{b + 1}{2}$.
	If $-1<b\leq 0$, $a>-1$ and $a+b+1>0$, then
	\begin{eqnarray}  \nonumber
		\lefteqn{E((\zeta(s+T)-\zeta(T))(\zeta(t+T)-\zeta(T)))}\\ \nonumber
		&=&\frac{1}{1-b}\left[\int_0^{s+T}u^{a}\left((T+t-u)^{b}+(T+s-u)^{b}-(2T+s+t-2u)^{b}\right)du\right.\\ \nonumber
		&+&\int_{0}^{T}u^{a}\left((T-u)^{b}+(T-u)^{b}-(2T-2u)^{b}\right) du
		-\int_{0}^{T}u^{a}\left((T+t-u)^{b}+(T-u)^{b}-(2T+t-2u)^{b}\right) du\\ \nonumber
		&-&\left.\int_{0}^{T}u^{a}\left((T+s-u)^{b}+(T-u)^{b}-(2T+s-2u)^{b}\right) \, du\right]\\ \nonumber
		&=&\frac{1}{1-b}\left[\int_{T}^{T+s}u^{a}\left((T+t-u)^{b}+(T+s-u)^{b}\right) \, du-\int_{T}^{T+s}u^{a}(2T+t+s-2u)^{b} \, du\right]\\       \nonumber
		&+&\frac{1}{1-b}\left[\int_{0}^{T}u^{a}\left((2T+s-2u)^{b}-(2T-2u)^{b}\right) \, du-\int_{0}^{T}u^{a}\left((2T+t+s-2u)^{b}-(2T+t-2u)^{b}\right) \, du\right]\\ \label{84}
		&=&H_1+H_2.
	\end{eqnarray}
	We are going to deal separately with each term in (\ref{84}). Changing variables $u-T\to u$ we get
	\begin{eqnarray}\label{nq1}
		\lefteqn{ T^{-a}\int_{T}^{T+s}u^{a}\left((T+t-u)^{b}+(T+s-u)^{b}\right)\,du}\\ \nonumber
		&=& \int_{0}^{s}\left(1+\frac{u}{T}\right)^{a}\left((t-u)^{b}+(s-u)^{b}\right)\,du
		\overset{T\to \infty}{\to} \int_{0}^{s}\left((t-u)^{b}+(s-u)^{b}\right)du
		=\frac{t^{b+1}+s^{b+1}-(t-s)^{b+1}}{b+1},
	\end{eqnarray}
	and
	\begin{equation}\label{nq2}
		T^{-a}\int_{T}^{T+s} u^{a}(2T+s+t-2u)^{b}du=\int_{0}^{s}(1+\frac{u}{T})^{a}(s+t-2u)^{b}du
		\overset{T\to \infty}{\to}  \frac{(s+t)^{b+1}-(t-s)^{b+1}}{2(b+1)}.
	\end{equation}
	From (\ref{nq1}) and (\ref{nq2}) we deduce that
	\begin{equation}
		\lim_{T\to \infty}T^{-a}H_1=\frac{1}{(1-b)(b+1)}\Bigg(s^{b+1}+t^{b+1}-\frac{1}{2}\Bigg((t-s)^{b+1}+(s+t)^{b+1}\Bigg)\Bigg).
	\end{equation}
	The term $H_2$ equals $bJ_2$ above, and  in this case we also obtain convergence (\ref{J_2}). Thus, the 
	process $\LL\{T^{-\alpha/2}(\zeta(t+T)-\zeta(T)),\ t\ge0\RR\}$
	converges as $T\to\infty$ to a fractional Brownian motion with parameter $(b + 1)/2$. This concludes the proof.
	\ter

	\noindent{{\bf
			Acknowledgment}\ The authors are grateful to an anonymous referee for her/his
		thorough revision of our paper, and for suggesting a number of arguments that significantly shortened and clarified several proofs.}

\end{document}